%% file: ND_partial0.tex
\documentclass[12pt]{amsart}

\usepackage{amssymb,amsmath,amsthm}

\newtheorem{theorem}{Theorem}[section]
\newtheorem{lemma}[theorem]{Lemma}
\newtheorem{prop}[theorem]{Proposition}
\newtheorem{cor}[theorem]{Corollary}

\def \R{\mathbb{R}}

\def \Rn{\mathbb{R}^n}
\def \Rno{\mathbb{R}^n_0}
\def \Rnone{\mathbb{R}^n_1}
\def \RnI{\mathbb{R}^{n+1}}
\def \RnIp{\mathbb{R}^{n+1}_{+}}
\def \Rniip{\mathbb{R}^{n+1}_{1+}}
\def \RnIB{\mathbb{R}^{n+1} \setminus B_1}

\def \a{\alpha}
\def \b{\beta}
\def \d{\delta}
\def \g{\gamma}

\def \e{\varepsilon}
\def \ph{\varphi}
\def \th{\theta}
\def \s{\sigma}

\def \L{\mathcal{L}}
\def \x{\chi}
\def \z{\zeta}

\def \nd{\partial_{\nu}}
\def \Om{\Omega}

\def \Lap{\triangle}
\def \grad{\nabla}

\def \S{\mathcal{S}}
\def \Czinf{C^{\infty}_0}

\def \Omt{\tilde{\Omega}}

\def \half{\frac{1}{2}}

\def \Lph{\mathcal{L}_{\ph}}
\def \Lphq{\mathcal{L}_{q,\ph}}
\def \Lphe{\mathcal{L}_{\ph,\e}}

\def \Lphet{\tilde{\mathcal{L}}_{\ph, \e}}

\numberwithin{equation}{section}

\begin{document}

\title[Partial Data Neumann-to-Dirichlet Map]{Partial Data for the Neumann-to-Dirichlet Map}

\author[Chung]{Francis J. Chung}
\address{Department of Mathematics and Statistics, University of Jyv\"{a}skyl\"{a}, Jyv\"{a}skyl\"{a}, Finland}



\begin{abstract}
We show that measurements of the Neumann-to-Dirichlet map on a certain part of the boundary of a domain in $\R^N$, $N \geq 3$, for inputs with support restricted to the other part, determine an electric potential on that domain.  Given a convexity condition on the domain, either the set on which measurements are taken, or the set on which input functions are supported, can be made to be arbitrarily small.  The result is analogous to the result by Kenig, Sj\"{o}strand, and Uhlmann for the Dirichlet-to-Neumann map.  The main new ingredient in the proof is a Carleman estimate for the Schr\"{o}dinger operator with appropriate boundary conditions.  
\end{abstract}

\maketitle

%log version
\section{Introduction}

\input{1Intro.tex}

\section{Using the Carleman Estimate}

\input{2Pfof1.tex}

\section{An Initial Carleman Estimate}

\input{3InitCarl.tex}

\section{The Flat Case}

\input{4FlatCase.tex}

\section{The Operators}

\input{5Operators.tex}

\section{The Linear Graph Case}

\input{6LinGraphCase.tex}

\section{Proof of the Linear Carleman Estimate}

\input{7PfCarl.tex}

\section{Logarithmic Operators}

\input{8LogOperators}

\section{The Logarithmic Case}

\input{9LogCase}

\section{Complex Geometrical Optics Solutions}

\input{10CGOsolns.tex}

\input{11Refs.tex}

\end{document}

%% file: 1Intro.tex
%log version

Consider the Euclidean space $\RnI$, $n \geq 2$, and suppose $\Om$ is a smooth bounded domain in $\RnI$.  Now suppose that $q \in L^{\infty}(\Om)$ is such that the problem
\begin{equation}
\begin{split}
(-\Lap + q)u &= 0 \mbox{ in } \Om \\
\partial_{\nu} u &= g \mbox { on } \partial \Om
\end{split}
\end{equation}
has a unique solution $u \in H^1(\Om)$ for every $g \in H^{-\half}(\partial \Om)$.  Then $q$ defines a Neumann-to-Dirichlet (ND) map $N_q: H^{-\half}(\partial \Om) \rightarrow H^{\half}(\partial \Om)$ by $N_q(g) = u|_{\partial \Om}$.

The basic inverse problem here is whether $N_{q}$ determines $q$.  This question is related to the corresponding question for the Dirichlet-to-Neumann (DN) map, which has been studied in several papers.  Notably, Sylvester and Uhlmann proved uniqueness for the DN problem in ~\cite{SU}, and Nachman gave a reconstruction method in ~\cite{N}.  For the Neumann-to-Dirichlet map, the fact that $N_{q}$ determines $q$ is a consequence of the argument in ~\cite{SU}.  

A more difficult question is whether partial knowledge of $N_{q}$ determines $q$.  Some recent papers on this subject have been written for the two-dimensional case.  In ~\cite{IUY2}, Imanuvilov, Uhlmann, and Yamamoto in ~\cite{IUY2} proved that measuring $N_{q}$ on arbitrary open domains determines $q$ for $q \in W^{1,p}$, $p >2$.  A slightly different problem, where assumptions are made on the potential in the neighbourhood of the boundary, was addressed by Hyv\"{o}nen, Piiroinen, and Seiskari, in ~\cite{HPS}. Earlier work by Imanuvilov, Uhlmann, and Yamamoto was also done for the DN map in two dimensions in ~\cite{IUY1}.  The results of ~\cite{IUY1} were then generalized to Riemannian surfaces by Guillarmou and Tzou in ~\cite{GT}.  

For three or higher dimensions, the only work known to this author on the partial data ND map problem comes from Isakov, who proved in ~\cite{I} that subsets of the boundary which coincide with a hyperplane or hypersphere may be ignored in the measurements, for both the ND and DN problems. For more general subsets of the boundary, in three or more dimensions, on the DN problem, there are several previous results.  Bukhgeim and Uhlmann in ~\cite{BU} and Kenig, Sj\"{o}strand,  and Uhlmann in ~\cite{KSU} show for the DN problem, roughly speaking, that measurements of the Dirichlet-to-Neumann map on certain parts of the boundary, using input functions whose supports are contained in the other part, determine $q$.  An improvement on these results is also given by Kenig and Salo in ~\cite{KeSa}.  Dos Santos Ferreira, Kenig, Sj\"{o}strand, and Uhlmann in ~\cite{DKSU}, Knudsen and Salo in ~\cite{KS}, and the present author in ~\cite{C} also provide similar results for the magnetic Schr\"{o}dinger equation, with a first order term.  A more comprehensive survey of progress in partial data problems of this sort can be found in ~\cite{KeSa_survey}.  

In this paper we will prove results analogous to those in ~\cite{BU} and ~\cite{KSU} for the Neumann-to-Dirichlet problem.  

At least part of the motivation for studying this question is to understand the case of partial data inverse problems for systems of equations, for which multiple types of boundary-data-to-boundary-data maps can exist.  Examples of these kinds of results can be found in work of Caro, Ola, and Salo for the Maxwell equations in ~\cite{COS} and in work of Salo and Tzou in ~\cite{ST}, for the Dirac equations. 

We can now state the main results.  Recall that $\Om \in R^{n+1}$, where $n \geq 2$.  If $\ph$ is smooth in a neighbourhood of $\Om$, define
\begin{equation*}
\begin{split}
\partial \Om_{+} &= \{ p \in \partial \Om | \partial_{\nu} \ph(p) \geq 0 \} \\
\partial \Om_{-} &= \{ p \in \partial \Om | \partial_{\nu} \ph(p) \leq 0 \} \\
\end{split}
\end{equation*}
where $\nu$ is the outward unit normal at $p$.  

\begin{theorem}\label{linearIPthm}
Let $q_1, q_2$ be in $L^{\infty}(\Om)$ such that $N_{q_1}$ and $N_{q_2}$ are defined. Let $\ph(x) = x_{n+1}$, and define $\partial \Om_{+}$ and $\partial \Om_{-}$  with respect to this choice of $\ph$.  Now let $\Gamma \subset \partial \Om$ be a neighbourhood of $\partial \Om_{+}$, and $Z \subset \partial \Om$ be a neighbourhood of $\partial \Om_{-}$.  Suppose 
\[
N_{q_1} g |_{\Gamma} = N_{q_2} g |_{\Gamma}
\]
for all $g \in H^{-\half}(\partial \Om)$ with support contained in $Z$.  Then $q_1 = q_2$ on $\Om$.
\end{theorem}

\begin{theorem}\label{logIPthm}
Let $q_1, q_2$ be in $L^{\infty}(\Om)$ such that $N_{q_1}$ and $N_{q_2}$ are defined. Let $p \in \RnI$ be outside the convex hull of $\overline{\Om}$, and let $\ph(x) = \pm \log |x-p|$. Define $\partial \Om_{+}$ and $\partial \Om_{-}$  with respect to the choice of $\ph$.  Now let $\Gamma \subset \partial \Om$ be a neighbourhood of $\partial \Om_{+}$, and $Z \subset \partial \Om$ be a neighbourhood of $\partial \Om_{-}$.  Suppose 
\[
N_{q_1} g |_{\Gamma} = N_{q_2} g |_{\Gamma}
\]
for all $g \in H^{-\half}(\partial \Om)$ with support contained in $Z$.  Then $q_1 = q_2$ on $\Om$.
\end{theorem}

A few remarks are in order.  First note that if $\Om$ is strictly convex (convex, and tangent planes at boundary points intersect the boundary in exactly one point) then Theorem \ref{logIPthm}, with the choice of $\ph = -\log|x-p|$, implies that the set on which the Neumann-to-Dirichlet maps are measured can be made arbitrarily small, by proper choice of $p$, provided the input functions are allowed to have support on a large part of the boundary.  On the other hand, choosing $\ph = +\log|x-p|$ implies that the set on which the input functions are supported can be arbitrarily small, provided one can measure the Neumann-to-Dirichlet map on a large subset of the boundary.  

Secondly, note that strictly speaking, Theorem \ref{linearIPthm} can be proved as a corollary of Theorem \ref{logIPthm} by choosing $p$ very far away from $\Om$.  However, for the sake of clarity, it will be easier to explain the proofs first in the case in which $\ph$ is linear, and then describe the proofs for the logarithmic cases in light of this explanation.  

Thirdly, these theorems imply a Neumann-to-Dirichlet result for the conductivity equation. If $\gamma \in C^2(\overline{\Om})$ is strictly positive, $g \in H^{-\half}(\partial \Om)$, and $u \in H^1(\Om)$ solves
\begin{eqnarray*}
\grad \cdot (\gamma \grad u) &=& 0 \mbox{ in } \Om \\
 \gamma \partial_{\nu} u|_{\partial \Om} &=& g
\end{eqnarray*}
then we can define $N_{\gamma}$ for the conductivity problem as the map sending $g$ to $u|_{\partial \Om}$.  If $\partial_{\nu} \gamma = 0$, and $q = \frac{\Lap \sqrt{\gamma}}{\sqrt{\gamma}}$, then by a change of variables (see ~\cite{SU}) 
\[
N_{q}(f) = \gamma^{\half} N_{\gamma}(\gamma^{\half} f).
\]
Therefore, suppose that $\gamma_1$ and $\gamma_2$ are such that $\partial_{\nu} \gamma_1 = \partial_{\nu} \gamma_2 = 0$ and $\gamma_1 = \gamma_2$ on the boundary of $\Om$.  Then if 
\[
N_{\gamma_1} g |_{\Gamma} = N_{\gamma_2} g |_{\Gamma}
\]
for all $g \in H^{-\half}(\partial \Om)$ with support in $Z$, where $\Gamma$ and $Z$ are as in Theorem \ref{linearIPthm} or \ref{logIPthm}, then $\gamma_1 = \gamma_2$.  

The partial data results for the Dirichlet-to-Neumann problem for $(-\Lap + q)$, discussed above, apply to the conductivity problem in a somewhat stronger fashion; see ~\cite{KSU}, for example.  

Note that the Dirichlet-to-Neumann problem for the conductivity equation has been the subject of much study in its own right.  In the case of $C^2$ conductivity in three and higher dimensions, the change of variables alluded to above, together with the work of Kohn and Vogelius in ~\cite{KV}, means that results in ~\cite{SU}, ~\cite{BU}, and ~\cite{KSU}, among others, apply to the conductivity equation as well.  Better regularity results have also been given for the conductivity equation: in three and higher dimensions, Haberman and Tataru have given a result for $W^{1, \infty}$ conductivity in the full data case in ~\cite{HT}, and in the two dimensional case, Astala and P\"{a}iv\"{a}rinta solved the Dirichlet-to-Neumann problem for $L^{\infty}$ conductivity in ~\cite{AP}.  In addition, Zhang has given a partial data result for less regular conductivities in three and higher dimensions in ~\cite{Zh}. 

The proofs of Theorem \ref{linearIPthm} and \ref{logIPthm} rely on a Carleman estimate, which will be stated as a theorem in its own right.  Let $h > 0$, and for a given choice of $\ph$, define
\[
\Lphq = e^{\frac{\ph}{h}}h^2 (-\Lap + q) e^{-\frac{\ph}{h}}  
\]

\begin{theorem}\label{mainCarl}
Choose $\ph$ to be as in Theorem \ref{linearIPthm} or \ref{logIPthm}. Define $\partial \Om_{+}$ and $\partial \Om_{-}$ with respect to that choice of $\ph$, and let $\Gamma \subset \partial \Om$ be a neighbourhood of $\partial \Om_{+}$.  Let $\Gamma^c$ denote $\partial \Om \setminus \Gamma$.  Now there exists $h_0 > 0$ such that for all $0 < h < h_0$, 
\begin{equation}\label{theCarl}
h \|w\|^2_{L^2(\Gamma^c)} + h\|h\grad_t w\|_{L^2(\Gamma^c)}^2 + h^2\|w\|^2_{L^2(\Om)} + h^2\|h\grad w\|^2_{L^2(\Om)} \lesssim \| \Lphq w \|^2_{L^2(\Om)}
\end{equation}
whenever $w \in H^2(\Om)$ with 
\begin{equation}\label{BC}
\begin{split}
w, \partial_{\nu} w &= 0 \mbox{ on } \Gamma \\
h\partial_{\nu} (e^{-\frac{\ph}{h}} w) &= h\s e^{-\frac{\ph}{h}} w \mbox{ on } \Gamma^c.
\end{split}
\end{equation}
for some smooth function $\s$ bounded independently of $h$ on $\Om$. Here $\grad_t$ represents the tangential part of the gradient along the boundary.  
\end{theorem}

The constant implied in the $\lesssim$ sign is independent of $h$.  For the remainder of this paper, inequalities of the form $A(h,w) \lesssim B(h,w)$ should be interpreted to mean that there exist constants $h_0 > 0$ and $C > 0$, with $C$ independent of $h_0$, such that for all $0 < h < h_0$, $A(h,w) \leq CB(h,w)$.

Note that the estimate \eqref{theCarl} can be rewritten as 
\[
h \|w\|^2_{H^1_s(\Gamma^c)} + h^2\|w\|^2_{H^1_s(\Om)} \lesssim \| \Lphq w \|^2_{L^2(\Om)}
\]
where $H^1_s$ is the semiclassical Sobolev space with semiclassical parameter $h$.  For the rest of the paper we'll adopt the convention that Sobolev spaces are meant to be the semiclassical variety, and express the Carleman estimate this way.  For a reference on semiclassical analysis, see ~\cite{Z}.

Theorem \ref{mainCarl} allows us to construct complex geometrical optics (CGO) solutions to the problem $(-\Lap + q)u = 0$ with Neumann data vanishing on $Z^c$.  We can describe these solutions by the following proposition.  

\begin{prop}\label{boundarysolns}
If $\ph$ and $Z^c$ are defined as in either of Theorems \ref{linearIPthm} or \ref{logIPthm}, then there exists a solution $u \in H^1(\Om)$ of the problem 
\begin{eqnarray*}
(-\Lap + q) u &=& 0 \mbox{ on } \Om\\
\partial_{\nu} u|_{Z^c} &=& 0 
\end{eqnarray*}
of the form $u = e^{\frac{1}{h}(-\ph + i \psi)}(a+r)$, where $\psi$ and $a$ are smooth functions with bounds independent of $h$; $\psi$ is a solution to the eikonal equation $\grad \ph \cdot \grad \psi = 0,  |\grad \psi| = |\grad \ph|$; and $\|r\|_{L^2(\Om)} \leq O(h^{\half})$.
\end{prop}

In particular, $\psi$ and $a$ are as in the CGO solutions in ~\cite{BU} and ~\cite{KSU}, for $\ph$ linear and $\ph$ logarithmic, respectively.  

The rest of the paper is structured as follows.  In Section 2, we'll see how Theorem \ref{mainCarl} and Proposition \ref{boundarysolns} are used to prove Theorem \ref{linearIPthm} and Theorem \ref{logIPthm}.  In Section 3, we'll prove an initial version of the Carleman estimate with extra terms on the right hand side of the inequality.  Sections 4-7 are then devoted to making these extra terms go away in the linear case, where $\ph(x) = x_{n+1}$, as in Theorem \ref{linearIPthm}.  In Sections 8-9, we'll see how the proof is modified to deal with the logarithmic case, where $\ph = \pm \log |x-p|$, as in Theorem \ref{logIPthm}.  Finally, Section 10 is devoted to the proof of Proposition \ref{boundarysolns}.  
\vspace{4mm}

\noindent \textbf{Acknowledgements}  The author would like to thank Mikko Salo for introducing him to this problem, for sharing the idea behind Proposition \ref{preCarl}, for reading over the manuscript, and for several other helpful conversations.  This research was partially supported by the Academy of Finland.  Part of this work was also done at the University of Chicago, and here the author would also like to thank Carlos Kenig for his time and support.

%% file: 2Pfof1.tex
%log version
Given the Carleman estimate in Theorem \ref{mainCarl}, and the CGO solutions guaranteed by Proposition \ref{boundarysolns}, the proofs of Theorem \ref{linearIPthm} and \ref{logIPthm} follow mostly from the arguments in ~\cite{BU} and ~\cite{KSU}.  First suppose that 
\[
u_1 = e^{\frac{-\ph + i\psi_1}{h}}(a_1 + r_1)
\]
is a CGO solution to
\begin{eqnarray*}
(-\Lap + q_1) u_1 &=& 0 \mbox{ on } \Om\\
\partial_{\nu} u_1 |_{Z^c} &=& 0, 
\end{eqnarray*}
as obtained from Proposition \ref{boundarysolns}.  Let 
\[
u_2 = e^{\frac{\ph + i\psi_2}{h}}(a_2 + r_2)
\]
be a standard CGO solution to $(-\Lap + q_2)u_2 = 0$.  Here $a_2$, $\psi_2$ and $r_2$ have the equivalent properties as for their counterparts in $u_1$, but nothing is guaranteed about the boundary behaviour of $u_2$.  Details can be found in ~\cite{SU}, ~\cite{KSU}, or ~\cite{DKSU}.  In fact, using the argument behind Proposition 2.4 in ~\cite{DKSU}, we can obtain $H^2$ regularity for $r_2$, and $\|r_2\|_{H^2(\Om)} = O(h).$

Now define $w \in H^1(\Om)$ to be the solution to
\begin{equation*}
\begin{split}
(-\Lap + q_2)w &= 0 \mbox{ on } \Om \\
\partial_{\nu} w|_{\partial \Om} &= \partial_{\nu} u_1 |_{\partial \Om} 
\end{split}
\end{equation*}
Then
\begin{eqnarray*}
& & \int_{\partial \Om} (N_{q_1} - N_{q_2}) (\partial_{\nu} u_1)\partial_{\nu}u_2 dS \\
&=& \int_{\partial \Om} (u_1 - w)\partial_{\nu}u_2 dS \\
&=& \int_{\Om} (u_1 - w)\Lap u_2 dV - \int_{\Om} \Lap(u_1 - w) u_2 dV \\
&=& \int_{\Om} (q_2 - q_1) u_1 u_2 dV. 
\end{eqnarray*}
Therefore if $N_{q_1}g = N_{q_2}g$ on $\Gamma$, for $g \in H^{-\half}(\partial \Om)$ with support in $Z$, then  
\begin{equation}\label{intid}
\int_{\Gamma^c} (u_1 - w) \partial_{\nu} u_2 dS = \int_{\Om} (q_2 - q_1)u_1 u_2 dV.
\end{equation}
Now $\|r_j\|_{L^2(\Om)} \leq O(h^{\half})$, so in the limit as $h \rightarrow 0$, the right side of \eqref{intid} becomes 
\[
\lim_{h \rightarrow 0} \int_{\Om} (q_2 - q_1) e^{i\frac{\psi_1 + \psi_2}{h}}a_1 a_2 dV.
\] 
Now consider the left side of \eqref{intid}.  
\[
\left| \int_{\Gamma^c} (u_1 -w) \partial_{\nu} u_2 dS \right| \leq h^{\half}\|e^{\frac{\ph}{h}}(u_1 - w)\|_{L^2(\Gamma^c)} \cdot h^{-\half}\|e^{-\frac{\ph}{h}}\partial_{\nu} u_2\|_{L^2(\Gamma^c)}.
\] 
The expression $e^{-\frac{\ph}{h}}\partial_{\nu} u_2$ can be written out as 
\[
e^{\frac{i\psi_2}{h}}\partial_{\nu}(a_2 + r_2) + h^{-1}\partial_{\nu}(\ph + i\psi_2) e^{\frac{i\psi_2}{h}}(a_2 + r_2).
\]
Since $a_2$ is smooth and bounded independently of $h$, 
\[
|e^{\frac{i\psi_2}{h}}\partial_{\nu}a_2 + h^{-1}\partial_{\nu}(\ph + i\psi_2) e^{\frac{i\psi_2}{h}}a_2| = O(h^{-1}).
\]
Now since $\|r_2\|_{H^2(\Om)} = O(h)$, we have that $\| \partial_{\nu}r_2\|_{L^2(\Gamma^c)} = O(h^{-\half})$ and $\|r_2\|_{L^2(\Gamma^c)} = O(h^{\half})$, so the expression
\[
h^{-\half}\|e^{-\frac{\ph}{h}}\partial_{\nu} u_2\|_{L^2(\Gamma^c)} 
\]
is $O(h^{-\frac{3}{2}})$.  Meanwhile, 
\[
\partial_{\nu}(e^{-\frac{\ph}{h}}e^{\frac{\ph}{h}}(u_1 - w))|_{\Gamma^c} = \partial_{\nu}(u_1 - w)|_{\Gamma^c} = 0
\] 
by definition, and $u_1 - w = 0$ on $\Gamma$ since $N_{q_1} = N_{q_2}$ there, so 
\[
\partial_{\nu} e^{\frac{\ph}{h}}(u_1 - w) = 0
\]
on $\Gamma$ also.  Therefore $e^{\frac{\ph}{h}}(u_1 - w)$ satisfies \eqref{BC}, and so the Carleman estimate applies to the first factor:
\begin{equation*}
\begin{split}
h^{\half}\|e^{\frac{\ph}{h}}(u_1 - w)\|_{L^2(\Gamma^c)} &\lesssim \|\L_{\ph,q_2} e^{\frac{\ph}{h}}(u_1 - w) \|_{L^2(\Om)} \\
                                                        &= h^2 \|e^{\frac{\ph}{h}}(-\Lap + q_2) (u_1 - w)\|_{L^2(\Om)} \\
                                                        &= h^2 \|e^{\frac{\ph}{h}}(-\Lap + q_2) u_1 \|_{L^2(\Om)} \\
                                                        &= h^2 \|e^{\frac{\ph}{h}}(-q_1 + q_2) u_1 \|_{L^2(\Om)} \\
                                                        &= h^2 \|(q_2 - q_1)e^{\frac{i\psi_1}{h}}(a_1 + r_1) \|_{L^2(\Om)} \\
\end{split}
\end{equation*}
Thus the first factor is $O(h^2)$, and so the left side of \eqref{intid} is $O(h^{\half})$.  Therefore in the limit as $h \rightarrow 0$, \eqref{intid} becomes
\[
\lim_{h \rightarrow 0} \int_{\Om} (q_2 - q_1) e^{i\frac{\psi_1 + \psi_2}{h}}a_1 a_2 dV = 0,
\]
and now it follows that $q_2 = q_1$ from the arguments in ~\cite{BU}, in the case that $\ph$ is linear, or by the arguments in ~\cite{KSU}, in the case that $\ph$ is logarithmic.  Therefore it remains only to prove Theorem \ref{mainCarl} and Proposition \ref{boundarysolns}.

%% file: 3InitCarl.tex
%log version
Let $h, \e > 0$, and define 
\[
\Lph = e^{\frac{\ph}{h}}h^2 \Lap e^{-\frac{\ph}{h}}
\]
and
\[
\Lphe = e^{\frac{\ph^2}{2\e}}\Lph e^{-\frac{\ph^2}{2\e}}.
\]
To begin, we'll prove the following Carleman estimate.  

\begin{prop}\label{preCarl}
Suppose $w \in H^2(\Om)$ satisfies the boundary conditions \eqref{BC}.  Then
\[
h^{\half}\|w\|_{L^2(\Gamma^c)} + \frac{h}{\sqrt{\e}}\|w\|_{H^1(\Om)} \lesssim \|\Lphe w \|_{L^2(\Om)} + h^{\half} \|h\grad_t w\|_{L^2(\Gamma^c)}.
\]
Here $H^1$ refers to the semiclassical Sobolev spaces with semiclassical parameter $h$, and $\grad_t$ is the tangential part of the gradient at $\partial \Om$.  
\end{prop}

Any Sobolev spaces that appear for the remainder of the paper are meant to be the semiclassical ones.  

\begin{proof}

The proof of this proposition follows the ideas from ~\cite{DKSU} quite closely, but with different boundary terms.  We can begin by writing $\Lphe$ out explicitly as 
\[
\Lphe = h^2\Lap + |\grad \ph_c|^2 - h \Lap \ph_c - 2\grad \ph_c \cdot h \grad
\]
where $\ph_c$ is the convexified version of $\ph$, 
\[
\ph_c = \ph + \frac{h \ph^2}{2\e}.
\]

Define the operators $P$ and $iQ$ by 
\[
P = h^2\Lap + |\grad \ph_c|^2
\]
and
\[
iQ = - h \Lap \ph_c - 2\grad \ph_c \cdot h\grad
\]
so 
\[
\Lphe = P + iQ.
\]
Then
\[
\| \Lphe w\|_{L^2(\Om)}^2 = \| Pw\|_{L^2(\Om)}^2 + \|Qw\|_{L^2(\Om)}^2 + (Pw,iQw) + (iQw, Pw),
\]
where $( \cdot, \cdot)$ denotes the $L^2$ inner product on $\Om$.  Integrating by parts,
\begin{eqnarray*}
\| \Lphe w\|_{L^2(\Om)}^2 &=& \| Pw\|_{L^2(\Om)}^2 + \|Qw\|_{L^2(\Om)}^2 + (i[P,Q]w,w) \\
                          & & + (Pw, -2h(\partial_{\nu} \ph_c) w)_{\partial \Om} -h^2(\partial_{\nu} iQw,w)_{\partial \Om} +h^2(iQw, \partial_{\nu}w)_{\partial \Om}.
\end{eqnarray*}

%***********************************************************************************************************************
We will first consider the nonboundary terms on the right hand side.  As in the proof of Proposition 4.1 from ~\cite{KeSa}, 
\[
i([P,Q]u|u) = 4\frac{h^2}{\e}\| (1 + h\e^{-1} \ph) u \|_{L^2}^2 + h(Q \beta Q u| u) + h^2 (Ru|u),
\]
where $R$ is a second order semiclassical differential operator whose coefficients are uniformly bounded in $h$ and $\e$, and $\beta = (h/\e)(1 + h\ph/\e)^{-2}$. Integration by parts gives that 
\[
h(Q \beta Q u| u) = h(\beta Qu|Qu) + \frac{2h^2}{i} ((\partial_{\nu} \ph_c) \beta Q u|u)_{\partial \Om}.
\]
We can write $Q$ at $\partial \Om$ as 
\[
Qu = ih \Lap \ph_c u + 2\partial_{\nu} \ph_c \partial_{\nu} u + (\grad \ph_c)_t \cdot ih\grad_t u
\]
where $(\grad \ph_c)_t$ and $\grad_t$ represent the tangential parts of $\grad \ph_c)$ and $\grad$, respectively.  From the boundary conditions on u, we get that $\|h \partial_{\nu} u\|_{\partial \Om} \lesssim \|u\|_{L^2(\partial \Om)}$, so 
\begin{eqnarray*}
|h(Q \beta Q u| u)| &\lesssim& \frac{h^2}{\e}\|Qu\|_{L^2}^2 + \frac{h^3}{\e}\|Qu\|_{L^2(\partial \Om)}^2 + \frac{h^3}{\e}\|u\|_{L^2(\partial \Om)}^2 \\
                    &\lesssim& \frac{h^2}{\e}\|Qu\|_{L^2}^2 + \frac{h^3}{\e}\|u\|_{H^1(\partial \Om)}^2. \\
\end{eqnarray*}
Similarly, 
\[
h^2 (Ru|u) \lesssim h^2\|u\|_{H^1}^2 + h^3\|u\|_{H^1(\partial \Om)}^2.
\]
Thus
\begin{equation}\label{iPQu}
i([P,Q]u|u) \gtrsim \frac{h^2}{\e}\|u\|_{L^2}^2 - \frac{h^2}{\e}\|Qu\|_{L^2}^2 - h^2\|u\|_{H^1}^2 - h^3\e^{-1}\|u\|_{H^1(\partial \Om)}^2
\end{equation}
for small enough $h$.  Now 
\begin{eqnarray*}
\|h\grad u\|_{L^2}^2 &=& (-h^2 \Lap u|u) + h^2(\partial_{\nu} u|u)_{\partial \Om} \\
                     &=& (-P u|u) + (|\grad \ph_c|^2 u|u) + h^2(\partial_{\nu} u|u)_{\partial \Om}.
\end{eqnarray*}
Using Cauchy-Schwartz, and invoking the boundary condition for the last term,
\[
h^2\|h\grad u\|_{L^2}^2 \lesssim \frac{1}{K}\|Pu\|_{L^2}^2+Kh^4\|u\|_{L^2}^2 + h^2\|u\|_{L^2}^2 + h^3\|u\|_{L^2(\partial \Om)}^2,
\]
or
\begin{equation}\label{Pu2}
\|Pu\|_{L^2}^2 \gtrsim Kh^2\|h\grad u\|_{L^2}^2 - K^2 h^4\|u\|_{L^2}^2 - Kh^2\|u\|_{L^2}^2 - Kh^3\|u\|_{L^2(\partial \Om)}^2.
\end{equation}
Therefore by choosing $K \sim \frac{1}{M \e}$ and taking $h$ small enough, and $M$ large enough, we get
\[
\|Pu\|^2 + \|Qu\|^2 + (i[P,Q]u|u) \gtrsim \frac{h^2}{\e}\|u\|_{H^1}^2 - \frac{h^3}{\e}\|u\|_{H^1(\partial \Om)}^2
\]
by combining \eqref{iPQu} and \eqref{Pu2}.  As a reminder, $H^1$ here indicates the semiclassical Sobolev space.  

%***********************************************************************************************************************

Therefore it remains only to understand the boundary terms that remain.  Note that \eqref{BC} implies that all of the boundary terms vanish on $\Gamma$.  On $\Gamma^c$, we have 
\[
h\partial_{\nu} w = (\partial_{\nu} \ph + h \s )w,
\]
so the boundary terms are
\[
(Pw, -2h(\partial_{\nu} \ph_c) w)_{\partial \Om} -h^2(\partial_{\nu} iQw,w)_{\partial \Om} +h(iQw, (\partial_{\nu} \ph + h\s)w)_{\partial \Om}.
\]
Now we can choose $U_1, \ldots, U_N \subset \RnI$ to be an open cover of $\partial \Om$ such that we can use boundary normal coordinates on each $U_m$.  Then if $\x_1, \ldots, \x_N$ is a partition of unity subordinate to the cover $U_1, \ldots, U_N$, and $w_m = \chi_m w$, we can rewrite the boundary terms as the sum over $m$ of
\[
(Pw, -2h(\partial_{\nu} \ph_c) w_m)_{U_m \cap \partial \Om} -h^2(\partial_{\nu} iQw,w_m)_{U_m \cap \partial \Om} +h(iQw, (\partial_{\nu} \ph + h\s)w_m)_{U_m \cap \partial \Om}.
\]
Now we can employ boundary normal coordinates in each $U_m$.  Let $\Rno$ be the hyperplane on which $x_{n+1} = 0$.  We have diffeomorphisms $U_m \mapsto V_m \subset \RnI$ which map $\partial \Om \cap U_m$ to $V_m \cap \Rno$, send $\nu$ to $e_{n+1}$, and induce a metric $g_m$ on $V_m$ of the form
\[g_m = \left( \begin{array}{cccc}
  &        &   & 0      \\
  & g_{m,0}&   & \vdots \\
  &        &   & 0      \\
0 & \ldots & 0 & 1      \\ \end{array} \right)\] 

Then we can write each integral over $U_m \cap \partial \Om$ as an integral over $V_m \cap \Rno$.  If we use $w$, $\ph$, $\ph_c$, and $\s$ to refer also to their pushforwards under the coordinate map, then the last term, $h(iQw, (\partial_{\nu} \ph + h\s)w_m)_{U_m \cap \partial \Om},$ becomes
\begin{eqnarray*}
& & -2h((\partial_{n+1} \ph_c) h\partial_{n+1}w + (\partial_j \ph_c) g_{m,0}^{jk} h\partial_k w, (\partial_{n+1} \ph + h\s)w_m a_m )_{V_m \cap \Rno} \\
&-& h^2((\Lap_{g_m}\ph_c)w, (\partial_{n+1} \ph + h\s)w_m a_m )_{V_m \cap \Rno}.
\end{eqnarray*}
The second term becomes
\begin{eqnarray*}
& & 2h((\partial_{n+1} \ph_c) h^2\partial_{n+1}^2 w + (\partial_j\ph_c) g_{m,0}^{jk} h^2\partial_{k} \partial_{n+1} w, w_m a_m )_{V_m \cap \Rno} \\
&+& 2h((h\partial_{n+1}^2 \ph_c) h \partial_{n+1} w + h\partial_{n+1}(\partial_j \ph_c g_{m,0}^{jk})h\partial_k w, w_m a_m )_{V_m \cap \Rno} \\
&+& h(h(\Lap_{g_m}\ph_c)h\partial_{n+1}w + h^2\partial_{n+1}(\Lap_{g_m}\ph_c)w, w_m a_m )_{V_m \cap \Rno},
\end{eqnarray*}
and the first term becomes
\begin{eqnarray*}
& & (h^2 \partial_{n+1}^2 w + g_{m,0}^{jk}h^2\partial_j\partial_k w, -2h(\partial_{n+1} \ph_c) w_m a_m )_{V_m \cap \Rno} \\
&+& (|\grad_{g_m} \ph_c|^2 w, -2h(\partial_{n+1} \ph_c) w_m a_m )_{V_m \cap \Rno} \\
&+& (|g_m|^{-\half} h\partial_j(|g_m|^{\half} g_{m,0}^{jk})h\partial_k w, -2h(\partial_{n+1} \ph_c) w_m a_m )_{V_m \cap \Rno} \\
\end{eqnarray*}
Here $g_{m,0}^{jk}$ refer to the indices of the inverse of $g_{m,0}$, and the summation convention is used from $1$ to $n$.  $\Lap_{g_m}$ and $\grad_{g_m}$ are meant to be the Laplace-Beltrami and gradient operators, respectively, for the metric $g_m$, and $a_m =\sqrt{|g_m|}$ is the factor generated by the change of variables.

Adding these together, we get
\begin{eqnarray*}
& &-2h((\partial_{n+1} \ph) h\partial_{n+1}w, (\partial_{n+1} \ph)w_m a_m )_{V_m \cap \Rno} \\
&-&2h((\partial_j \ph_c) g_{m,0}^{jk} h\partial_k w, (\partial_{n+1} \ph)w_m a_m )_{V_m \cap \Rno} \\
&+&2h((\partial_{n+1} \ph_c) h^2\partial_{n+1}^2 w , w_m a_m )_{V_m \cap \Rno} \\
&+&2h((\partial_j\ph_c) g_{m,0}^{jk} h^2\partial_{k} \partial_{n+1} w, w_m a_m )_{V_m \cap \Rno} \\
&-&2h(h^2 \partial_{n+1}^2 w,(\partial_{n+1} \ph_c) w_m a_m )_{V_m \cap \Rno} \\
&-&2h(g_{m,0}^{jk}h^2\partial_j\partial_k w, (\partial_{n+1} \ph_c) w_m a_m )_{V_m \cap \Rno} \\
&-&2h(|\grad_{g_m} \ph_c|^2 w, (\partial_{n+1} \ph_c) w_m a_m )_{V_m \cap \Rno} \\
&+&(O(h^2) w,w_m)_{V_m \cap \Rno} + \sum_k (O(h^2) h\partial_k w, w_m)_{V_m \cap \Rno} + (O(h^2) h\partial_{n+1} w, w_m)_{V_m \cap \Rno}. \\
\end{eqnarray*}
The third and fifth terms cancel, and using the boundary conditions on $w$, the error terms in the last line can be bounded by
\begin{equation}\label{stuffbound}
O(h^2)\|w\|^2_{L^2(V_m \cap \Rno)}+ O(h)\sum_k \|h\partial_k w\|^2_{L^2(V_m \cap \Rno)}.
\end{equation}
Now integration by parts, together with the boundary conditions on $w$, shows that the second and fourth terms cancel up to this error as well.   Finally, integration by parts in the sixth term, shows that it is also bounded by \eqref{stuffbound}.  Therefore up to this error, we have
\[
-2h((\partial_{n+1} \ph)((\partial_{n+1} \ph)^2 + |\grad_{g_m} \ph|^2)w, w_m a_m )_{V_m \cap \Rno}.
\]
Translating back to $U_m \cap \partial \Om$, we have
\[
-2h((\partial_{\nu} \ph)((\partial_{\nu} \ph)^2 + |\grad \ph|^2)w, w_m)_{U_m \cap \partial \Om}.
\]
up to an error bounded by 
\[
O(h^2)\|w\|^2_{L^2(\partial \Om)}+ O(h)\|h\grad_t w\|^2_{L^2(\partial \Om)}.
\]
Adding together the boundary terms for each $U_m$, we have
\begin{eqnarray*}
& & -2h((\partial_{\nu} \ph)((\partial_{\nu} \ph)^2 + |\grad \ph|^2)w, w)_{\partial \Om} \\
&=& -2h((\partial_{\nu} \ph)((\partial_{\nu} \ph)^2 + |\grad \ph|^2)w, w)_{\Gamma^c} \\
\end{eqnarray*}
plus an error bounded by 
\[
O(h^2)\|w\|^2_{L^2(\Gamma^c)}+ O(h)\|h\grad_t w\|^2_{L^2(\Gamma^c)}.
\]

Note that since $\Gamma$ is a neighbourhood of $\partial \Om_{+}$, $-(\partial_{\nu} \ph)$ is bounded below by some positive number on $\Gamma^c$, so we can write this as
\[
2h\|\sqrt{|\partial_{\nu} \ph|((\partial_{\nu} \ph)^2 + |\grad \ph|^2)} w\|^2_{L^2(\Gamma^c)}.
\]

Therefore the equation
\begin{eqnarray*}
\| \Lphe w\|_{L^2(\Om)}^2 &=& \| Pw\|_{L^2(\Om)}^2 + \|Qw\|_{L^2(\Om)}^2 + (i[P,Q]w,w) \\
                          & & + (Pw, -2(h\partial_{\nu} \ph_c) w)_{\partial \Om} -h^2(\partial_{\nu} iQw,w)_{\partial \Om} +h^2(iQw, \partial_{\nu}w)_{\partial \Om}.
\end{eqnarray*}
becomes 
\begin{eqnarray*}
& &       \| \Lphe w\|_{L^2(\Om)}^2 + h^2\|w\|_{L^2(\Gamma^c)}^2 + h\|h\grad_t w\|_{L^2(\Gamma^c)}^2 \\
&\gtrsim& \frac{h^2}{\e}\|w\|_{H^1(\Om)} + h\|\sqrt{|\partial_{\nu} \ph|(|\grad \ph|^2 + (\partial_{\nu} \ph)^2)}w\|_{L^2(\Gamma^c)}^2 -\frac{h^3}{\e}\|u\|_{H^1(\Gamma^c)}^2.
\end{eqnarray*}
The last term on the right side can be absorbed into the boundary terms on the left side, for small enough $h$.  Then since $|\partial_{\nu} \ph|$ is bounded below by some positive number on $\Gamma^c$, for small enough $h$, the second term on the left side can be absorbed into the right side.  Therefore we end up with 
\[
\| \Lphe w\|_{L^2(\Om)}^2 + h\|h\grad_t w\|_{L^2(\Gamma^c)}^2 \gtrsim \frac{h^2}{\e}\|w\|^2_{H^1(\Om)} + h\|w\|_{L^2(\Gamma^c)}^2.
\]
The proposition now follows.

\end{proof}

%% file: 4FlatCase.tex
%log version
Now we need to make the boundary term on the left side of the inequality go away.  For the next four sections, including this one, we'll now assume that $\ph$ is the linear weight.  To differentiate the $\ph$ direction from the others, we'll choose coordinates $(x,y)$ on $\RnI$ where $x \in \Rn$ and $y \in \R$, and choose $\ph(x,y) = y$.  

In order to understand the basic idea of the rest of the argument, we'll present it first in the case where $q=0$, $\Gamma^c$ is contained in the hypersurface $\{y = 0\}$, and
\[
\Om \subset \RnIp = \{(x,y) \in \Rn \times \R | y \geq 0 \},
\]
and the function $\s$ that appears in \eqref{BC} is zero.  
We'll let $\Rno$ denote the boundary of $\RnIp$.  By the methods used to prove Proposition \ref{preCarl}, we can get 
\begin{equation}\label{simpleCarl}
h^{\half}\|w\|_{L^2(\Gamma^c)} + h\|w\|_{H^1(\Om)} \lesssim \|\Lph w \|_{L^2(\Om)} + h^{\half} \|h \grad_x w\|_{L^2(\Gamma^c)}.
\end{equation}
for $w \in H^2(\Om)$ satisfying \eqref{BC}.  Now suppose $w \in C^{\infty}(\overline{\Om})$ such that $w$ satisfies \eqref{BC} with $\s=0$, and $w \equiv 0$ in a neighbourhood of $\Gamma$.  Then we can extend $w$ by $0$ to the rest of $\RnIp$ to obtain a function in $\S(\RnIp)$, defined as the space of restrictions to $\RnIp$ of Schwartz functions on $\RnI$.  Then we can write
\[
h^{\half}\|w\|_{L^2(\Rno)} + h\|w\|_{H^1(\RnIp)} \lesssim \|\Lph w \|_{L^2(\RnIp)} + h^{\half} \|w\|_{\dot{H}^1(\Rno)}.
\]
Now we can take the Fourier transform in the $x$ variables only.  We'll use the notation $\hat{w} = \hat{w}(\xi,y)$ to denote the semiclassical Fourier transform of $w$ in the $x$ variables.  

$\Lph w$ can be written out as 
\[
\Lph w = h^2\partial_y^2 w - 2h\partial_y w + (1 + h^2 \Lap_x)w,
\]
so taking Fourier transforms in the $x$ variables gives 
\begin{equation*}
\begin{split}
\widehat{\Lph w} &= (h^2\partial_y^2 - 2 h \partial_y + 1 - |\xi|^2) \hat{w} \\
                 &= (h\partial_y - (1 + |\xi|))(h\partial_y - (1 - |\xi|)) \hat{w} \\
\end{split}
\end{equation*}
Let $T_{\psi}$ denote the operator obtained by using $\psi$ as a Fourier multiplier.  Then we can express this as
\[
\Lph w = (h\partial_y - T_{1+|\xi|})(h\partial_y - T_{1-|\xi|})w.
\]

Now the operator $(h\partial_y - T_{1+|\xi|})$ has the property that 
\[
\|(h\partial_y - T_{1+|\xi|})v\|_{L^2(\RnIp)} \simeq \|v\|_{H^1(\RnIp)}.
\]
(For a proof, see Lemma \ref{bddness}.)  Therefore
\[
\| \Lph w \|_{L^2(\RnIp)} \simeq \|(h\partial_y - T_{1-|\xi|})w\|_{H^1(\RnIp)}.
\]
By the semiclassical trace theorem (see below), 
\[
\|(h\partial_y - T_{1-|\xi|})w\|_{H^1(\RnIp)} \gtrsim h^{\half}\|(h\partial_y - T_{1-|\xi|})w\|_{L^2(\Rno)}.
\]
Now by \eqref{BC}, $w$ satisfies the boundary condition $h\partial_y w = w$ on $\Rno$.  Therefore
\begin{equation*}
\begin{split}
\|(h\partial_y - T_{1-|\xi|})w\|_{L^2(\Rno)} &= \|w - T_{1-|\xi|}w\|_{L^2(\Rno)} \\
                                             &= \|T_{|\xi|}w\|_{L^2(\Rno)} \\
                                             &\simeq \|w\|_{\dot{H}^1(\Rno)} \\
\end{split}
\end{equation*}
Putting this all together gives that 
\[
h^{\half}\|w\|_{\dot{H}^1(\Rno)} \lesssim \| \Lph w\|_{L^2(\RnIp)},
\]
and therefore, by \eqref{simpleCarl}, we end up with
\[
h^{\half}\|w\|_{H^1(\Gamma^c)} + h\|w\|_{H^1(\Om)} \lesssim \|\Lph w \|_{L^2(\Om)}
\]
for any $w \in C^{\infty}(\overline{\Om})$ satisfying \eqref{BC}, where $w \equiv 0$ in a neighbourhood of $\Gamma$.  A density argument then finishes the proof.  

For completeness, we give a short proof of the semiclassical trace formula mentioned above, based on Lemma \ref{bddness}.
\begin{lemma}
Let $v \in \S(\RnIp)$.  Then the map $v(x,y) \mapsto v(x,0)$ extends to a bounded map $\tau: H^1(\RnIp) \rightarrow L^2(\Rno)$ with
\[
h^{\half}\|\tau(v)\|_{L^2(\Rno)} \lesssim \|v\|_{H^1(\RnIp)}.
\]
\end{lemma}

\begin{proof}
Suppose $v \in \S(\RnIp)$.  Define a function $u$ on $\RnIp$ by
\[
\hat{u}(\xi,y) = -\langle \xi \rangle \hat{v}(\xi,0)e^{-\frac{\langle \xi \rangle y}{h}}.
\]
Then $u \in L^2(\RnIp)$, with
\begin{equation}\label{ufunction}
\|u\|_{L^2(\RnIp)} \lesssim h^{\half}\| v \|_{H^{\half}(\Rno)}.
\end{equation}
Now 
\[
(u, (h\partial_y - T_{\langle \xi \rangle})v)_{\RnIp} = ((-h\partial_y - T_{\langle \xi \rangle})u, v)_{\RnIp} - h(u,v)_{\Rno}.
\]
by integration by parts.  Now by definition of $u$, we have that $(h\partial_y + T_{\langle \xi \rangle})u = 0$ and $-(u,v)_{\Rno} \simeq \|v\|^2_{H^{\half}(\Rno)}$.  Therefore
\[
(u, (h\partial_y - T_{\langle \xi \rangle})v)_{\RnIp} \simeq h\|v\|^2_{H^{\half}(\Rno)}.
\]
Then Cauchy-Schwarz gives
\[
\|u\|_{L^2(\RnIp)}\|(h\partial_y - T_{\langle \xi \rangle})v\|_{L^2(\RnIp)} \gtrsim h\|v\|^2_{H^{\half}(\Rno)}.
\]
By Lemma \ref{bddness}, 
\[
\|u\|_{L^2(\RnIp)}\|v\|_{H^1(\RnIp)} \gtrsim h \|v\|^2_{H^{\half}(\Rno)}.
\]
Dividing through by $\|u\|_{L^2(\RnIp)}$, and using \eqref{ufunction}, gives the desired result.  

\end{proof}

To make this idea work in the general case, we'll first concentrate on the case where $\Gamma^c$ is contained in a graph $\{ y = f(x)\}$, with $\Om \subset \{y > f(x)\}$, and $\grad f$ is close to a constant $K$.  Then we can change variables to flatten out $\Gamma^c$ , and attempt to carry out the program above.  The change of variables has the effect of perturbing the operator $\Lph$, so the factoring becomes more delicate, but the argument can still be carried through.  Finally, these graph estimates can be glued together to give Theorem \ref{mainCarl}.

%% file: 5Operators.tex
%log version
First however, we should introduce a family of operators for use in the linear case.  Suppose $F(\xi)$ is a complex valued function on $\Rn$, with the properties that $|F(\xi)|, \mathrm{Re} F(\xi) \simeq 1 + |\xi|$.  

Then for $u \in \S(\RnIp)$, define $J$ by
\[
\widehat{Ju}(\xi,y) = (F(\xi) + h\partial_y)\hat{u}(\xi,y).
\]
This has adjoint $J^*$ defined by 
\[
\widehat{J^*u}(\xi,y) = (\overline{F}(\xi) - h\partial_y)\hat{u}(\xi,y).
\]
These operators have right inverses given by
\[
\widehat{J^{-1}u} = \frac{1}{h}\int_0^y \hat{u}(\xi,t) e^{F(\xi)\frac{t-y}{h}}dt
\]
and
\[
\widehat{J^{*-1}u} = \frac{1}{h}\int_y^\infty \hat{u}(\xi,t) e^{\overline{F}(\xi)\frac{y-t}{h}}dt.
\]
Now we have the following boundedness result.

\begin{lemma}\label{bddness}
The operators $J$, $J^{*}$, $J^{-1}$, and $J^{*-1}$, initially defined on $\S(\RnIp)$, extend to bounded operators
\[
J, J^{*}: H^1(\RnIp) \rightarrow L^2(\RnIp)
\]
and
\[
J^{-1}, J^{*-1}: L^2(\RnIp) \rightarrow H^1(\RnIp)
\]
Moreover, these extensions for $J^{*}$ and $J^{*-1}$ are isomorphisms.
\end{lemma}

\begin{proof}
Consider $J$ first, and suppose $u \in \S(\RnIp)$.
\begin{equation*}
\begin{split}
\|Ju\|^2_{L^2(\RnIp)} &\simeq h^{-n}\|\widehat{Ju}\|^2_{L^2(\RnIp)} \\
                    &\leq h^{-n}\|h\partial_y \hat{u}\|^2_{L^2(\RnIp)} + h^{-n}\|F(\xi)\hat{u}\|^2_{L^2(\RnIp)} \\
                    &\lesssim h^{-n}\|h\partial_y \hat{u}\|^2_{L^2(\RnIp)} + h^{-n}\|(1 + |\xi|) \hat{u}\|^2_{L^2(\RnIp)} \\
                    &\lesssim \|u\|^2_{H^1(\RnIp)} \\
\end{split}
\end{equation*}
By a density argument, $J$ now extends to a bounded operator $J: H^1(\RnIp) \rightarrow L^2(\RnIp)$.  The argument for $J^{*}$ is similar.

Now consider $J^{-1}$.  If $u \in \S(\RnIp)$,
\begin{equation*}
\begin{split}
\|\widehat{J^{-1}u}\|^2_{L^2(\RnIp)} &=    \frac{1}{h^2}\int_{\Rn} \int_0^\infty \left| \int_0^y \hat{u}(\xi,t) e^{F(\xi)\frac{t-y}{h}}dt \right|^2 dy\, d\xi. \\
                                     &=    \frac{1}{h^2}\int_{\Rn} \int_0^\infty \left| \int_{-\infty}^y \hat{u}(\xi,t) e^{F(\xi)\frac{t-y}{h}}dt \right|^2 dy\, d\xi. \\
                                     &\leq \frac{1}{h^2}\int_{\Rn} \int_{-\infty}^\infty \left| \int_{-\infty}^0 \hat{u}(\xi,t+y) e^{F(\xi)\frac{t}{h}}dt \right|^2 dy\, d\xi. \\
\end{split}                                     
\end{equation*}
Then by Minkowski's inequality,
\[
\|\widehat{J^{-1}u}\|^2_{L^2(\RnIp)} \leq \frac{1}{h^2}\int_{\Rn} \left( \int_{-\infty}^0 \left( \int_{-\infty}^\infty |\hat{u}(\xi,t+y)|^2 e^{2\mathrm{Re} F(\xi)\frac{t}{h}}dy \right)^{\half}dt\right)^2  d\xi.
\]
Therefore
\begin{equation*}
\begin{split}
\|\widehat{J^{-1}u}\|^2_{L^2(\RnIp)} &\leq \frac{1}{h^2}\int_{\Rn} \left( \int_{-\infty}^\infty |\hat{u}(\xi,y)|^2 dy \left( \int_{-\infty}^0 e^{\mathrm{Re} F(\xi)\frac{t}{h}}dt \right)^2 \right) d\xi \\
                                     &=  \int_{\Rn} \int_{-\infty}^\infty \left| \frac{\hat{u}(\xi,y)}{\mathrm{Re} F(\xi)} \right|^2 dy \, d\xi \\
                                     &\lesssim \|\hat{u}\|^2_{L^2(\RnIp)}. \\
\end{split}                                     
\end{equation*}
Similarly, 
\[
\|\xi_j \widehat{J^{-1}u}\|^2_{L^2(\RnIp)} \lesssim \|\hat{u}\|^2_{L^2(\RnIp)}.
\]
Finally,
\[
h\partial_y \widehat{J^{-1}u} = -F(\xi)\widehat{J^{-1}u} + \hat{u}
\]
so
\[
\|h\partial_y \widehat{J^{-1}u}\|^2_{L^2(\RnIp)} \lesssim \|\hat{u}\|^2_{L^2(\RnIp)} 
\]
also. Putting all of this together gives
\[
\|J^{-1}u\|_{H^1(\RnIp)} \leq \|u\|_{L^2(\RnIp)}
\]
as desired, and a density argument shows that $J^{-1}$ extends to a bounded operator from $L^2(\RnIp)$ to $H^1(\RnIp)$.  The proof for $J^{*-1}$ is similar.  

Now we need to show that the extensions for $J^{*}$ and $J^{*-1}$ are isomorphisms.  To do this, note that if $u \in \S(\RnIp)$, then
\[
J^{*} J^{*-1} u = u,
\]
and (using integration by parts)
\[
J^{*-1} J^{*} u = u.  
\]
Then a density argument finishes the proof.  
\end{proof}

Note that similar mapping properties hold between $H^1(\RnIp)$ and $H^2(\RnIp)$, by the same reasoning.  

We'll need to record one more operator fact in this section.  

Let $m,k \in \mathbb{Z}$, with $m,k\geq 0$.  Suppose $a(x,\xi,y)$ are smooth functions on $\Rn \times \Rn \times \R$ that satisfy the bounds
\[
\|\partial_x^\b \partial_\xi^{\a} \partial_y^j a(x,\xi,y)\| \leq C_{\a,\b} (1 + |\xi|)^{m-|\a|}
\]
for all multiindices $\a$ and $\b$, and for $0 \leq j \leq k$.  In other words, each $\partial_y^j a(x,\xi,y)$ is a symbol on $\Rn$ of order $m$, with bounds uniform in $y$, for $0 \leq j \leq k$.  Then we can define an operator $A$ on Schwartz functions in $\RnI$ by applying the pseudodifferential operator on $\Rn$ with symbol $a(x,\xi,y)$, defined by the Kohn-Nirenberg quantization, to $f(x,y)$ for each fixed $y$.  More generally, we can also define operators $A_j$ on Schwartz functions in $\RnI$ by applying the pseudodifferential operator on $\Rn$ with symbol $\partial_y^j a(x,\xi,y)$ to $f(x,y)$ for each fixed $y$, for $1 \leq j \leq k$.  

\begin{lemma}\label{flatops}
If $A$ is as above, then $A$ extends to a bounded operator from $H^{k+m}(\RnI)$ to $H^k(\RnI)$.
\end{lemma}

\begin{proof}
Suppose $f \in \S(\RnI)$.  Since $k \in \mathbb{Z}$, $k \geq 0$, 
\[
\| Af \|^2_{H^k(\R^{n+1})} = \sum_{0 \leq |\a| + j \leq k} \|h^{|\a| + j}\partial^{\a}_x \partial^j_y Af \|^2_{L^2(\R^{n+1})}.
\]
Now $\partial_y^j A(f)$ is a sum of terms of the form 
\[
A_{j_1}\partial_y^{j_2} f
\]
where $j_1 + j_2 = j \leq k$.  Therefore $\| Af \|^2_{H^k(\R^{n+1})}$ is bounded by a sum of terms of the form
\[
\|h^{|\a| + j_1 + j_2}\partial^{\a}_x A_{j_1} \partial^{j_2}_y f \|^2_{L^2(\R^{n+1})},
\]
where $|\a|+j_1+j_2 \leq k$.  Then 
\begin{equation*}
\begin{split}
\|h^{|\a| + j_1 + j_2}\partial^{\a}_x A_{j_1} \partial^{j_2}_y f \|^2_{L^2(\R^{n+1})} &= \int_{\R} \int_{\Rn} |h^{|\a| + j_1 + j_2}\partial^{\a}_x A_{j_1} \partial^{j_2}_y f|^2 dx \, dy \\
&\leq \int_{\R} \|h^{j_1 + j_2} A_{j_1} \partial^{j_2}_y f\|_{H^{|\a|}(\Rn)}^2 dy \\
\end{split}
\end{equation*}
Then by the boundedness of $A_{j_1}$, this is bounded above by
\[
\int_{\R} \|h^{j_2}\partial^{j_2}_y f\|_{H^{|\a|+m}(\Rn)}^2 dy,
\]
which in turn is bounded above by
\begin{equation*}
\begin{split}
\|h^{j_2}\partial^{j_2}_y f\|_{H^{|\a|+m}(\RnI)}^2 &\leq \|f\|_{H^{|\a|+m+j_2}(\RnI)}^2 \\
                                                   &\leq \|f\|_{H^{k+m}(\RnI)}^2
\end{split}                                                   
\end{equation*}
Therefore
\[
\| Af \|^2_{H^k(\R^{n+1})} \lesssim \|f\|_{H^{k+m}(\RnI)}^2.
\]
Then a density argument finishes the proof.

\end{proof}

%% file: 6LinGraphCase.tex
%log version

Suppose $f \in \Czinf(\Rn)$.  In this section, we'll examine the case where $\Om$ lies in the set $\{ y \geq f(x)\}$, and $\Gamma^c$ lies in the graph $\{ y = f(x) \}$.  

Here we can do the change of variables $(x,y) \mapsto (x, y - f(x))$.  Define $\tilde{\Om}$ and $\tilde{\Gamma}$ to be the images of $\Om$ and $\Gamma$ respectively, under this map. Note that $\{ y \geq f(x)\}$ maps to $\RnIp$, and $\Gamma^{c}$ maps to a subset of $\Rno$.  

For $\tilde{\Om}$ we can obtain the following proposition.  

\begin{prop}\label{tildepreCarl}
Suppose $w \in H^2(\Omt)$, and  
\begin{equation}\label{tildeBC}
\begin{split}
w, \partial_{\nu} w &= 0 \mbox{ on } \tilde{\Gamma} \\
h \partial_{y}w|_{\tilde{\Gamma}^c} &= \frac{w + \grad f \cdot h \grad_{x} w - h\s w}{1 + |\grad f|^2}.
\end{split}
\end{equation}
where $\s$ is smooth and bounded on $\Omt$.  
Then
\[
h^{\half}\|w\|_{L^2(\tilde{\Gamma}^c)} + \frac{h}{\sqrt{\e}}\|w\|_{H^1(\Omt)} \lesssim \|\Lphet w \|_{L^2(\Omt)} + h^{\half} \|h\grad_x w\|_{L^2(\tilde{\Gamma}^c)}
\]
where
\[
\Lphet = (1+|\grad f|^2)h^2\partial_y^2 - 2(\a + \grad f \cdot h\grad_x)h\partial_y + \a^2 + h^2\Lap_x
\]
and $\a = 1 + \frac{h}{\e}(y + f(x))$.  Note that on $\Omt$, $\a$ is very close to $1$.
\end{prop}

\begin{proof}
Suppose $w \in H^2(\Omt)$ satisfies \eqref{tildeBC}.  Let $v$ be the function on $\Om$ defined by $v(x,y) = w(x, y - f(x))$.  Then $v \in H^2(\Om)$, and $v$ satisfies \eqref{BC}.  Therefore by Proposition \ref{preCarl}, 
\[
h^{\half}\|v\|_{L^2(\Gamma^c)} + \frac{h}{\sqrt{\e}}\|v\|_{H^1(\Om)} \lesssim \|\Lphe v \|_{L^2(\Om)} + h^{\half} \|h\grad_t v\|_{L^2(\Gamma^c)}.
\]
Now by a change of variables,
\begin{equation*}
\begin{split}
\|v\|_{L^2(\Gamma^c)} &\simeq \|w\|_{L^2(\tilde{\Gamma}^c)}, \\
\|v\|_{H^1(\Om)} &\simeq \|w\|_{H^1(\Omt)}, \\
\end{split}
\end{equation*}
and
\[
\|h\grad_t v\|_{L^2(\Gamma^c)} \simeq \|h\grad_x w\|_{L^2(\tilde{\Gamma}^c)}.
\]
Moreover, 
\[
\left( \Lphe v \right)(x, y + f(x)) = \Lphet \left( w(x,y) \right) + hE_1 w(x,y)
\]
where $E_1$ is a first order semiclassical differential operator.  Therefore by a change of variables,
\[
\|\Lphe v \|_{L^2(\Om)} \lesssim \|\Lphet w \|_{L^2(\Omt)} + h\|w\|_{H^1(\Omt)}.
\]
Putting this all together gives
\[
h^{\half}\|w\|_{L^2(\tilde{\Gamma}^c)}+\frac{h}{\sqrt{\e}}\|w\|_{H^1(\Omt)} \lesssim \|\Lphet w\|_{L^2(\Omt)}+h^{\half} \|h\grad_x w\|_{L^2(\tilde{\Gamma}^c)}+ h\|w\|_{H^1(\Omt)}.
\]
For sufficiently small $\e$, the last term on the right side can be absorbed into the left side to give
\[
h^{\half}\|w\|_{L^2(\tilde{\Gamma}^c)}+\frac{h}{\sqrt{\e}}\|w\|_{H^1(\Omt)} \lesssim \|\Lphet w\|_{L^2(\Omt)}+h^{\half} \|h\grad_x w\|_{L^2(\tilde{\Gamma}^c)}
\]
as desired.  

\end{proof}

We need to find a way to bound the last term in the inequality by the other terms.  To do this, we'll split the last term into two separate parts, a small frequency and large frequency part.  

To simplify matters, we'll assume for the rest of this section that there are constants $K \in \Rn$ and $\d > 0$ such that $|\grad f(x) - K| < \d$ for all $x$ such that some $(x,y) \in \Om$.

Now, choose $m_2 > m_1 > 0$, and $\mu_1$ and $\mu_2$ such that 
\[
\frac{|K|}{\sqrt{1+|K|^2}} < \mu_1 < \mu_2 < \half + \frac{|K|}{2\sqrt{1+|K|^2}} < 1.
\]
The eventual choice of $\mu_j$ and $m_j$ will depend only on $K$.  

Define $\rho \in \Czinf(\Rn)$ such that $\rho(\xi) = 1$ if $|\xi| < \mu_1$ and $|K \cdot \xi| < m_1$, and $\rho(\xi) = 0$ if $|\xi| > \mu_2$ or $|K \cdot \xi| > m_2$.  

Now suppose $w \in C^{\infty}(\Omt)$ such that $w \equiv 0$ in a neighbourhood of $\tilde{\Gamma}$, and $w$ satisfies \eqref{tildeBC}.  We can extend $w$ by zero to the rest of $\RnIp$, so $w \in \S(\RnIp)$, as in the flat case.  Set $w_s = T_{\rho} w$ and $w_{\ell} = (1 - T_{\rho}) w$, so $w = w_s + w_{\ell}$.  Then by Proposition \ref{tildepreCarl}, 
\begin{equation}\label{freqsplit}
h^{\half}\|w\|_{L^2(\Rno)} + \frac{h}{\sqrt{\e}}\|w\|_{H^1(\RnIp)} \lesssim \|\Lphet w \|_{L^2(\RnIp)} + h^{\half} \|w_s\|_{\dot{H}^1(\Rno)} + h^{\half} \|w_{\ell}\|_{\dot{H}^1(\Rno)}.
\end{equation}
Now we'll examine each of the last two terms on the right side separately.  The next proposition will deal with the small frequency term.

\begin{prop}\label{smallprop}
Suppose $w$ is as above.  There exist choices of $m_1, m_2, \mu_1,$ and $\mu_2$, depending only on $K$, such that if $\d$ is small enough,
\[
h^{\half}\|w_s\|_{\dot{H}^1(\Rno)} \lesssim \|\Lphet w \|_{L^2(\RnIp)} + h\|w\|_{H^1(\RnIp)}.
\]
\end{prop}

Before proceeding to the proof, let's make some definitions.  If $V \in \Rn$, define $A_{\pm}(V, \xi)$ by
\[
A_{\pm}(V,\xi) = \frac{1 + iV \cdot \xi \pm \sqrt{(1+iV\cdot \xi)^2 - (1+|V|^2)(1-|\xi|^2)}}{1 + |V|^2},
\]
In other words, $A_{\pm}(V, \xi)$ are defined to be the roots of the polynomial
\[
(1 + |V|^2)X^2 - 2(1+ iV \cdot \xi)X + (1 - |\xi|^2)
\]
In the definition, we'll choose the branch of the square root which has non-negative real part, so the branch cut occurs on the negative real axis.  

\begin{proof}
Now consider the behaviour of $A_{\pm}(K,\xi)$ on the support of $\rho$, or equivalently, on the support of $\hat{w}_s$.  If $\eta > 0$, we can choose $\mu_2$ such that on the support of $\hat{w}_s$, 
\[
1 - (1+|K|^2)(1-|\xi|^2) < \eta.
\]
Then on the support of $\hat{w}_s$, the expression
\[
(1+iK\cdot \xi)^2 - (1+|K|^2)(1-|\xi|^2)
\]
has real part confined to the interval $[-K^2-m_2^2, \eta+m_2^2]$, and imaginary part confined to the interval $[-2m_2, 2m_2]$.  Therefore, by correct choice of $\eta$ and $m_2$, we can ensure
\[
\mathrm{Re}A_{\pm}(K, \xi) > \frac{1}{2(1+|K|^2)}.
\]
on the support of $\hat{w}_s$.  This allows us to fix the choice of $\mu_1, \mu_2, m_1,$ and $m_2$.  Note that the choices depend only on $K$, as promised.  

The bounds on $A_{\pm}(K, \xi)$ allow us to choose $F_{\pm}$ so that $F_{\pm} = A_{\pm}(K, \xi)$ on the support of $\hat{w}_s$, and $\mathrm{Re}F_{\pm}, |F_{\pm}| \simeq 1 + |\xi|$ on $\Rn$, with constant depending only on $K$.  Therefore $F_{+}$ and $F_{-}$ both satisfy the conditions on $F$ in Section 5.  It follows that the operators $h \partial_y - F_{+}$ and $h\partial_y - F_{-}$ both have the properties of $J^{*}$ in that section.  

Up until now, the operator $\Lphet$ has only been applied to functions supported in $\Omt$.  However, we can extend the coefficients of $\Lphet$ to $\RnIp$ while retaining the $|\grad f - K| < \d$ condition.  Then
\begin{eqnarray*}
\|\Lphet w_s\|_{L^2(\RnIp)} &=& \|((1+|\grad f|^2)h^2\partial_y^2 - 2(\a + \grad f \cdot h\grad_x)h\partial_y + \a^2 + h^2\Lap_x)w_s\|_{L^2(\RnIp)} \\
                            &\geq& \|((1+|K|^2)h^2\partial_y^2 - 2(1 + K \cdot h\grad_x)h\partial_y + 1 + h^2\Lap_x)w_s\|_{L^2(\RnIp)} \\
                            & & - C\delta \|w_s\|_{H^2(\RnIp)}\\
\end{eqnarray*}
for sufficiently small $h$.  Meanwhile, 
\begin{eqnarray*}
& & (1+|K|^2)(h\partial_y - T_{F_{+}})(h\partial_y - T_{F_{-}})w_s \\
&=& (1+|K|^2)(h^2 \partial_y^2 - T_{F_{+} + F_{-}} h\partial_y + T_{F_{+}F_{-}})w_s. \\
\end{eqnarray*}
Since $F_{\pm} = A_{\pm}(K, \xi)$ on the support of $\hat{w}_s$, this can be written as 
\begin{eqnarray*}
& & (1+|K|^2)(h^2 \partial_y^2 - T_{A_{+} + A_{-}}h\partial_y + T_{A_{+}A_{-}})w_s\\
&=& ((1+|K|^2)h^2\partial_y^2 - 2(1 + K \cdot h\grad_x)h\partial_y + 1 + h^2\Lap_x)w_s\\
\end{eqnarray*}
Therefore
\[
\|\Lphet w_s\|_{L^2(\RnIp)} \geq \|(h\partial_y - T_{F_{+}})(h\partial_y - T_{F_{-}})w_s \|_{L^2(\RnIp)} - C\delta\|w_s\|_{H^2(\RnIp)}.
\]
Now by the boundedness properties,
\[
\|(h\partial_y - T_{F_{+}})(h\partial_y - T_{F_{-}})w_s \|_{L^2(\RnIp)} \simeq \|w_s\|_{H^2(\RnIp)},
\]
so for small enough $\delta$, 
\[
\|\Lphet w_s\|_{L^2(\RnIp)} \gtrsim \|w_s\|_{H^2(\RnIp)}.
\]
Then by the semiclassical trace formula, 
\[
\|\Lphet w_s\|_{L^2(\RnIp)} \gtrsim h^{\half}\|w_s\|_{\dot{H}^1(\Rno)}.
\]
Finally, note that 
\begin{equation*}
\begin{split}
\|\Lphet w_s\|_{L^2(\RnIp)} &=        \|\Lphet T_{\rho} w\|_{L^2(\RnIp)}\\
                            &\lesssim \|(1 + |\grad f|^2)^{-1} \Lphet T_{\rho} w\|_{L^2(\RnIp)}\\
                            &\lesssim \|T_{\rho} (1 + |\grad f|^2)^{-1} \Lphet  w\|_{L^2(\RnIp)} + \|hE_1w\|_{L^2(\RnIp)}.
\end{split}
\end{equation*}
where $hE_1$ comes from the commutator of $T_{\rho}$ and $(1 + |\grad f|^2)^{-1} \Lphet$.  By Lemma \ref{flatops}, $E_1$ is bounded from $H^1(\RnIp)$ to $L^2(\RnIp)$, so
\[
\|\Lphet w_s\|_{L^2(\RnIp)} \lesssim \|\Lphet  w\|_{L^2(\RnIp)} + h\|w\|_{H^1(\RnIp)}.
\]
Therefore
\[
\|\Lphet  w\|_{L^2(\RnIp)} + h\|w\|_{H^1(\RnIp)} \gtrsim h^{\half}\|w_s\|_{\dot{H}^1(\Rno)}
\]
as desired.

\end{proof}

Now we have to deal with the large frequency term.  

\begin{prop}\label{largeprop}
Suppose $w$ is the extension by zero to $\RnIp$ of a function in $C^{\infty}(\Omt)$ which is $0$ in a neighbourhood of $\tilde{\Gamma}$, and satisfies \eqref{tildeBC}, and let $w_{\ell}$ be defined as above.  Then if $\d$ is small enough,
\[
h^{\half} \|w_{\ell}\|_{\dot{H}^1(\Rno)} \lesssim \|\Lphet w \|_{L^2(\RnIp)} + h\|w\|_{H^1(\RnIp)} + h^{\frac{3}{2}}\|w\|_{L^2(\Rno)}.
\]
\end{prop}

\begin{proof}

Suppose $V \in \Rn$.  Recall that we defined 
\[
A_{\pm}(V,\xi) = \frac{1 + iV \cdot \xi \pm \sqrt{(1+iV\cdot \xi)^2 - (1+|V|^2)(1-|\xi|^2)}}{1 + |V|^2},
\]
so $A_{\pm}(V,\xi)$ are roots of the polynomial
\[
(1 + |V|^2)X^2 - 2(1+ iV \cdot \xi)X + (1 - |\xi|^2).
\]
Now let's define
\[
A^{\e}_{\pm}(V,\xi) = \frac{\a + iV \cdot \xi \pm \sqrt{(\a+iV\cdot \xi)^2 - (1+|V|^2)(\a^2-|\xi|^2)}}{1 + |V|^2},
\]
so $A^{\e}_{\pm}(V,\xi)$ are the roots of the polynomial
\[
(1 + |V|^2)X^2 - 2(\a+ iV \cdot \xi)X + (\a^2 - |\xi|^2).
\]
(Recall that $\a$ is defined by $\a = 1 + \frac{h}{\e}(y + f(x))$.)  Again we'll use the branch of the square root with non-negative real part.  

Now set $\zeta \in C^{\infty}_0(\Rn)$ to be a smooth cutoff function such that $\zeta = 1$ if 
\[
|K \cdot \xi| < \half m_1 \mbox{ and } |\xi| < \half \frac{|K|}{\sqrt{1+|K|^2}} +  \half \mu_1, 
\]
and $\zeta = 0$ if $|K \cdot \xi| \geq m_1$ or $|\xi| \geq \mu_1$.  

Now define 
\[
G_{\pm}(V, \xi) = (1-\zeta)A_{\pm}(V,\xi) + \zeta
\]
and
\[
G^{\e}_{\pm}(V, \xi) = (1-\zeta)A^{\e}_{\pm}(V,\xi) + \zeta.
\]
Consider the singular support of $A^{\e}_{\pm}(\grad f,\xi)$.  These are smooth as functions of $x$ and $\xi$ except when the argument of the square root falls on the non-positive real axis.  This occurs when $\grad f \cdot \xi = 0$ and
\[
|\xi|^2 \leq \frac{\a^2|\grad f|^2}{1 + |\grad f|^2}.
\] 
Now for $\d$ sufficiently small, depending on $K$, this does not occur on the support of $1 -\z$.  Therefore
\[
G^{\e}_{\pm}(\grad f, \xi) = (1-\zeta)A^{\e}_{\pm}(\grad f,\xi) + \zeta  
\]
are smooth, and one can check that they are symbols of first order on $\Rn$.  

Then by properties of pseudodifferential operators,
\begin{eqnarray*}
& & (1+|\grad f|^2)(h\partial_y - T_{G^{\e}_{+}(\grad f, \xi)})(h\partial_y - T_{G^{\e}_{-}(\grad f, \xi)}) \\
&=& (1+|\grad f|^2)(h^2\partial_y^2 - T_{G^{\e}_{+}(\grad f, \xi) + G^{\e}_{-}(\grad f, \xi)}h\partial_y + T_{G^{\e}_{+}(\grad f, \xi)G^{\e}_{-}(\grad f, \xi)})+hE_1,
\end{eqnarray*}
where $E_1$ is bounded from $H^1(\RnIp)$ to $L^2(\RnIp)$.  
This last line can be written out as
\begin{equation*}
\begin{split}
&(1+|\grad f|^2)h^2\partial_y^2 - 2(\a + \grad f \cdot h \grad_x)h\partial_y T_{1 - \zeta}T_{1+\z} + (\a + h^2\Lap_x)T_{(1-\z)^2} \\
&+ hE_1 + T_{\zeta^2} -2 h\partial_y T_{\zeta},
\end{split}
\end{equation*}
by modfiying $E_1$ as necessary.  Now
\[
T_{1 - \zeta}w_{\ell} = w_{\ell}
\]
and
\[
T_{\zeta}w_{\ell} = 0,
\]
so
\begin{eqnarray*}
& & (1+|\grad f|^2)(h\partial_y - T_{G^{\e}_{+}(\grad f, \xi)})(h\partial_y - T_{G^{\e}_{-}(\grad f, \xi)})w_{\ell}\\
&=& \Lphet w_{\ell} - h E_1 w_{\ell}.
\end{eqnarray*}

Therefore
\begin{eqnarray*}
\|\Lphet w_{\ell} \|_{L^2(\RnIp)} &\gtrsim& \|(h\partial_y - T_{G^{\e}_{+}(\grad f, \xi)})(h\partial_y - T_{G^{\e}_{-}(\grad f, \xi)})w_{\ell}\|_{L^2(\RnIp)} \\
                                     & & - h\|w_{\ell}\|_{H^1(\RnIp)}.\\
\end{eqnarray*}
Now
\[
G^{\e}_{+}(\grad f, \xi) =  G_{+}(K, \xi) + (G^{\e}_{+}(\grad f, \xi) -  G_{+}(K, \xi)),
\]
and
\[
T_{G^{\e}_{+}(\grad f, \xi) -  G_{+}(K, \xi)}
\]
involves multiplication by functions bounded by $O(\d)$, so
\[
\|T_{G^{\e}_{+}(\grad f, \xi) -  G_{+}(K, \xi)}v\|_{L^2(\RnIp)} \lesssim \d \|v\|_{H^1(\RnIp)}.
\]
Therefore
\begin{eqnarray*}
\|\Lphet w_{\ell} \|_{L^2(\RnIp)} &\gtrsim& \|(h\partial_y - T_{G_{+}(K, \xi)})(h\partial_y - T_{G^{\e}_{-}(\grad f, \xi)})w_{\ell}\|_{L^2(\RnIp)} \\
                                     & & - h\|w_{\ell}\|_{H^1(\RnIp)} - \d\|(h\partial_y - T_{G^{\e}_{-}(\grad f, \xi)})w_{\ell}\|_{H^1(\RnIp)} .\\
\end{eqnarray*}
Now we can check that $G_{+}(K,\xi)$ satisfies the necessary properties of $F$ from Section 5, so
\begin{eqnarray*}
\|\Lphet w_{\ell} \|_{L^2(\RnIp)} &\gtrsim& \|(h\partial_y - T_{G^{\e}_{-}(\grad f, \xi)})w_{\ell}\|_{H^1(\RnIp)} \\
                                     & & - h\|w_{\ell}\|_{H^1(\RnIp)} - \d\|(h\partial_y - T_{G^{\e}_{-}(\grad f, \xi)})w_{\ell}\|_{H^1(\RnIp)} .\\
\end{eqnarray*}
Then for small enough $\d$,
\begin{equation*}
\begin{split}
\|\Lphet w_{\ell} \|_{L^2(\RnIp)} &\gtrsim \|(h\partial_y - T_{G^{\e}_{-}(\grad f, \xi)})w_{\ell}\|_{H^1(\RnIp)} - h\|w_{\ell}\|_{H^1(\RnIp)} \\
                                  &\gtrsim h^{\half}\|(h\partial_y - T_{G^{\e}_{-}(\grad f, \xi)})w_{\ell}\|_{L^2(\Rno)} - h\|w_{\ell}\|_{H^1(\RnIp)}.
\end{split}
\end{equation*}

Now by \eqref{tildeBC},
\[
h\partial_y w = \frac{w + \grad f \cdot h \grad_x w + h\s w}{1 + |\grad f|^2} 
\]
on $\Rno$, so
\[
h\partial_y w_{\ell} = \frac{w_{\ell} + \grad f \cdot \grad_x w_{\ell}}{1 + |\grad f|^2} + hE_0 w
\]
on $\Rno$, where $E_0$ is bounded from $L^2(\Rn)$ to $L^2(\Rn)$.  Therefore
\begin{eqnarray*}
\|\Lphet w_{\ell}\|_{L^2(\RnIp)} &\gtrsim& h^{\half}\left\| \frac{w_{\ell}+\grad f \cdot \grad_x w_{\ell}}{1+|\grad f|^2} -T_{G^{\e}_{-}(\grad f,\xi)}w_{\ell} \right\|_{L^2(\Rno)}\\
                                       & & - h\|w_{\ell}\|_{H^1(\RnIp)} - h^{\frac{3}{2}}\|w\|_{L^2(\Rno)} \\
                                 &\gtrsim& h^{\half} \|w_{\ell} \|_{\dot{H}^1(\Rno)} - h\|w_{\ell}\|_{H^1(\RnIp)} - h^{\frac{3}{2}}\|w\|_{L^2(\Rno)}      
\end{eqnarray*}

Now 
\[
\|w_{\ell}\|_{H^1(\RnIp)} \lesssim \|w\|_{H^1(\RnIp)}
\]
and
\[
\|\Lphet w_{\ell}\|_{L^2(\RnIp)} \lesssim \|\Lphet w \|_{L^2(\RnIp)} + h\|w\|_{H^1(\RnIp)}.
\]
Therefore
\[
\|\Lphet w \|_{L^2(\RnIp)} + h\|w\|_{H^1(\RnIp)} + h^{\frac{3}{2}}\|w\|_{L^2(\Rno)}  \gtrsim h^{\half} \|w_{\ell} \|_{\dot{H}^1(\Rno)}
\]
as desired.

\end{proof}

Now using the results of Propositions \ref{smallprop} and \ref{largeprop} in \eqref{freqsplit} gives
\[
h^{\half}\|w\|_{L^2(\Rno)} + h^{\half}\|w\|_{\dot{H}^1(\Gamma^c)}+ \frac{h}{\sqrt{\e}}\|w\|_{H^1(\RnIp)} \lesssim \|\Lphet w \|_{L^2(\RnIp)} +h\|w\|_{H^1(\RnIp)} + h^{\frac{3}{2}}\|w\|_{L^2(\Rno)}.
\]
The last two terms can be absorbed into the left side (for small enough $h$ and $\e$) to give
\[
h^{\half}\|w\|_{H^1(\Rno)} + \frac{h}{\sqrt{\e}}\|w\|_{H^1(\RnIp)} \lesssim \|\Lphet w \|_{L^2(\RnIp)}
\]
for $w \in C^{\infty}(\Omt)$ such that $w \equiv 0$ in a neighbourhood of $\tilde{\Gamma}$, and $w$ satisfies \eqref{tildeBC}.  A density argument and a change of variables then gives
\begin{equation}\label{graphest}
h^{\half} \|w\|_{H^1(\Gamma^c)} + \frac{h}{\sqrt{\e}}\|w\|_{H^1(\Om)} \lesssim \| \Lphe w \|_{L^2(\Om)}
\end{equation}
for all $w \in H^2(\Om)$ satisfying \eqref{BC}, in the case where $\Gamma^c$ coincides with a part of the graph $y = f(x)$, with $|\grad f - K| \leq \d$.  Note that the choice of $\d$ depends ultimately only on $K$.

%% file: 7PfCarl.tex
%log version
Now we need to remove the graph conditions on $\Gamma^c$.  Since $\Gamma$ is a neighbourhood of $\partial \Om_{+}$, in a small enough neighbourhood $U$ around any point on $\Gamma^c$, $\Gamma^c$ coincides locally with a subset of a graph of the form $y = f(x)$, with $\Om \cap U$ lying in the set $y > f(x)$, and $|\grad f - K| < \d$, where $K$ is some constant, and $\d$ is small enough for \eqref{graphest} to hold.  

Therefore we can let $\{U_1, \ldots U_m\}$ be a finite open cover of $\overline{\Om}$ such that each $\Om \cap U_j$ has smooth boundary, and each $\Gamma^c \cap U_j$ is either empty or represented as a graph of the form $y = f_j(x)$, with $|\grad f_j - K_j| < \d_j$, where $\d_j$ are small enough for 
\[
h^{\half} \|v_j\|_{L^2(\Gamma^c \cap U_j)} + \frac{h}{\sqrt{\e}}\|v_j\|_{H^1(U_j \cap \Om)} \lesssim \| \Lphe v_j \|_{L^2(U_j \cap \Om)}
\]
to hold for all $v_j \in H^2(\Om \cap U_j)$ such that 
\begin{equation}\label{jBC}
\begin{split}
v_j, \partial_{\nu} v_j &= 0 \mbox{ on } \partial (U_j \cap \Om) \setminus \Gamma^c  \\
h\partial_{\nu} (e^{-\frac{\ph}{h}} v_j) &= h\s e^{-\frac{\ph}{h}} v_j \mbox{ on } \Gamma^c \cap U_j.
\end{split}
\end{equation}

Now let $\chi_1, \ldots \chi_m$ be a partition of unity subordinate to $U_1, \ldots U_m$, and for $w \in H^2(\Om)$ satisfying \eqref{BC}, define $w_j = \chi_j w$.  Then if $\Gamma^c \cap U_j \neq \varnothing$, $w_j$ satisfies \eqref{jBC} for some $\s$, and so 
\[
h^{\half} \|w_j\|_{H^1(\Gamma^c \cap U_j)} + \frac{h}{\sqrt{\e}}\|w_j\|_{H^1(\Om)} \lesssim \| \Lphe w_j \|_{L^2(\Om)}.
\]
On the other hand, if $\Gamma^c \cap U_j = \varnothing$, then
\[
\frac{h}{\sqrt{\e}}\|w_j\|_{H^1(\Om)} \lesssim \| \Lphe w_j \|_{L^2(\Om)}
\]
just by applying Proposition \ref{preCarl}.  

Adding together these estimates gives
\[
h^{\half} \|w\|_{H^1(\Gamma^c)} + \frac{h}{\sqrt{\e}}\|w\|_{H^1(\Om)} \lesssim \sum_{j=1}^m \| \Lphe w_j \|_{L^2(\Om)}.
\]
Now each $\| \Lphe w_j \|_{L^2(\Om)} = \| \Lphe \chi_j w \|_{L^2(\Om)}$ is bounded by a constant times $\| \Lphe w \|_{L^2(\Om)} + h\|w\|_{H^1(\Om)}$, so
\[
h^{\half} \|w\|_{H^1(\Gamma^c)} + \frac{h}{\sqrt{\e}}\|w\|_{H^1(\Om)} \lesssim \| \Lphe w \|_{L^2(\Om)} + h\|w\|_{H^1(\Om)}.
\]
The last term on the right side can be absorbed into the left side for small enough $\e$, so
\[
h^{\half} \|w\|_{H^1(\Gamma^c)} + \frac{h}{\sqrt{\e}}\|w\|_{H^1(\Om)} \lesssim \| \Lphe w \|_{L^2(\Om)}.
\]

Now we can get rid of the $\e$ in the operator.  Note that if $w \in H^2(\Om)$, satisfies \eqref{BC}, then so does $e^{\frac{\ph^2}{2\e}}w$, albeit with a different choice of $\s$.  Therefore
\[
h^{\half} \|e^{\frac{\ph^2}{2\e}}w\|_{H^1(\Gamma^c)} + \frac{h}{\sqrt{\e}}\|e^{\frac{\ph^2}{2\e}}w\|_{H^1(\Om)} \lesssim \| e^{\frac{\ph^2}{2\e}} \Lph w \|_{L^2(\Om)},
\]
and since $e^{\frac{\ph^2}{2\e}}$ is bounded above and below on $\Om$, 
\[
h^{\half} \|w\|_{H^1(\Gamma^c)} + h\|w\|_{H^1(\Om)} \lesssim \| \Lph w \|_{L^2(\Om)}.
\]

Finally, 
\[
\| \Lphq w \|_{L^2(\Om)} \geq  \| \Lph w \|_{L^2(\Om)} - h^2\|qw\|_{L^2(\Om)},
\]
so absorbing the extra term into the left side,
\[
h^{\half} \|w\|_{H^1(\Gamma^c)} + h\|w\|_{H^1(\Om)} \lesssim \| \Lphq w \|_{L^2(\Om)}.
\]
This finishes the proof of Theorem \ref{mainCarl} in the linear case.

%% file: 8LogOperators.tex
%log version
Now we turn to the logarithmic case of Theorem \ref{mainCarl}.  First we'll need a set of operators for the logarithmic case to parallel those introduced in Section 5 for the linear case.  

Again suppose that $F(\xi)$ is a complex valued function on $\Rn$, with the properties that $|F(\xi)|, \mathrm{Re} F(\xi) \simeq 1 + |\xi|$.  Define $\Rniip = \{ (x,r)| x \in \Rn, r > 1 \}$, and define $\S(\Rniip)$  to be the space of restrictions to $\Rniip$ of Schwartz functions on $\RnI$.   

Now for $u \in \S(\Rniip)$, define $J_{\log} u$ by 
\[
\widehat{J_{\log} u}(r, \xi) = \left( \frac{F(\xi)}{r}+h\partial_r \right) \hat{u}(r, \xi).
\]
This operator has adjoint $J_{\log}^{*}$ given by
\[
\widehat{J_{\log}^{*} u}(r, \xi) = \left( \frac{\overline{F(\xi)}}{r}-h\partial_r \right) \hat{u}(r, \xi).
\]
These operators have right inverses defined by 
\[
\widehat{J_{\log}^{-1}u} (r,\xi) = h^{-1} \int_1^{r}\hat{u}(t,\xi) \left( \frac{t}{r} \right)^{\frac{F(\xi)}{h}} dt 
\]
and
\[
\widehat{J_{\log}^{*-1}u} (r,\xi) = h^{-1} \int_r^{\infty}\hat{u}(t,\xi) \left( \frac{r}{t} \right)^{\frac{\overline{F(\xi)}}{h}} dt.
\]
To obtain the analog of Lemma \ref{bddness}, we need to introduce the weighted Sobolev space $H^1_r(\Rniip)$, whose norm is defined by 
\[
\| u \|^2_{H^1_r(\Rniip)} = \left\| \frac{u}{r} \right\|^2_{L^2(\Rniip)} + \|h\partial_r u\|^2_{L^2(\Rniip)} + \left\| \frac{h}{r}\grad_{x} u \right\|^2_{L^2(\Rniip)}.
\]
Then we have the following boundedness results.  

\begin{lemma}\label{logbounds}
$J_{\log}, J_{\log}^{*},J_{\log}^{-1},$ and $J_{\log}^{*-1}$ extend as bounded maps
\[
J_{\log}, J_{\log}^{*}: H^1_r(\Rniip) \rightarrow L^2(\Rniip)
\]
and
\[
J_{\log}^{-1}, J_{\log}^{*-1}: L^2(\Rniip) \rightarrow H^1_r(\Rniip).
\]
Moreover, the extensions of $J_{\log}^{*}$ and $J_{\log}^{*-1}$ are isomorphisms.  
\end{lemma}

The proof follows the method used for Lemma \ref{bddness}.  In addition, these operators are identical to the ones introduced in ~\cite{C}, and the proof of this theorem is included in full there.

%% file: 9LogCase.tex
%log version
Now we can deal with the logarithmic cases of Theorem \ref{mainCarl}.  Let $p$ be outside the convex hull of $\overline{\Om}$.  By a change of coordinates on $\RnI$, we can assume that $p = 0$.  Then fix $\ph(x) = \log |x|$.  (We will deal with $\ph = -\log |x|$ later.)

As in the linear argument, we examine the graph case first.  Fix spherical coordinates $(\th, r)$ on $\RnI$, where $\th \in S^n$ and $r \in [0, \infty)$.  Suppose $f: S^n \rightarrow (0, \infty)$ is smooth, and consider the case where $\Om$ lies in the set $\{ r \geq f(\th)\}$, and $\Gamma^c$ lies in the graph $\{ r = f(\th) \}$.  Note that in these coordinates, $\ph = \log r$.  

Following Section 6, consider the change of variables 
\[
(\th,r) \mapsto \left( \th, \frac{r}{f(\th)} \right).  
\]
This maps $\Om$ to a subset of $\RnIB$, where $B_1$ is the unit ball in $\RnI$, centred at the origin, and maps $\Gamma^c$ to be a subset of $\partial B_1$.  Define $\tilde{\Om}'$ and $\tilde{\Gamma}'$ to be the images of $\Om$ and $\Gamma$ respectively, under this map. 

Then from Proposition \ref{preCarl} we obtain the following lemma for $\tilde{\Om}'$.  

\begin{lemma}\label{prelogtildepreCarl}
Suppose $w \in H^2(\Omt ')$, and  
\begin{equation*}
\begin{split}
w, \partial_{\nu} w &= 0 \mbox{ on } \tilde{\Gamma}' \\
h \partial_{r}w|_{\tilde{\Gamma} '^c} &= \frac{w + (\grad_{S^n} \log f) \cdot h \grad_{S^n} w - h\s w}{1 + |\grad_{S^n} \log f|^2} .
\end{split}
\end{equation*}
where $g$ is smooth and bounded on $\Omt '$.  
Then
\[
h^{\half}\|w\|_{L^2(\tilde{\Gamma}'^c)} + \frac{h}{\sqrt{\e}}\|w\|_{H^1(\Omt')} \lesssim \|\Lphet ' w \|_{L^2(\Omt')} + h^{\half} \|h\grad_{\th} w\|_{L^2(\tilde{\Gamma}'^c)}
\]
where
\[
\Lphet ' = (1+|\grad_{S^n} \log f|^2)h^2\partial_r^2 - \frac{2}{r}(\a + (\grad_{S^n} \log f) \cdot h\grad_{S^n})h\partial_y + \frac{1}{r^2}(\a^2 + h^2\Lap_{S^n})
\]
and $\a = 1 + \frac{h}{\e}\log(rf(\theta))$.  
\end{lemma}

In order to be able to take Fourier transforms, as in the linear case, we will make one more change of variables, using spherical coordinates, to make everything Euclidean.  

Note that the requirement that $p = 0$ is outside the closure of the convex hull of $\Om$ means that $\overline{\Om}$ lies on one side of a hyperplace through the origin.  Then the same must be true of $\Omt '$.  Therefore we can choose spherical coordinates $(r, \th_1, \ldots, \th_n)$ on $\RnIB$, 
\[
x_1 = r \cos \th_1, \, x_2 = r \sin \th_1 \cos \th_2, \ldots, x_{n+1} = r \sin \th_1 \cdots \sin \th_n,
\]
such that the change of variables $\Omt ' \rightarrow [1, \infty) \times (0, \pi) \times \ldots \times (0,\pi)$ given by taking spherical coordinates on $\Omt '$ is a diffeomorphism on $\Omt '$, which takes $\Omt '$ to a subset of $\Rniip$.  

Define $\Omt$ and $\tilde{\Gamma}$ to be the images of $\Omt '$ and $\tilde{\Gamma} '$ under this change of variables.  Note that $\tilde{\Gamma}^c$ is now a subset of the hyperplane $\Rnone := \{ y= 1 \}$. 

Now we have the following proposition.  
\begin{prop}\label{logtildepreCarl}
Suppose $w \in H^2(\Omt)$, and  
\begin{equation}\label{logtildeBC}
\begin{split}
w, \partial_{\nu} w &= 0 \mbox{ on } \tilde{\Gamma} \\
h \partial_{r}w|_{\tilde{\Gamma}^c} &= \frac{w + \b(x) \cdot h \grad_x w - h\s w}{1 + |\g(x) |^2}.
\end{split}
\end{equation}
where $\b$ and $\g$ are smooth vector fields on $\Rn$ such that $\b(x) \cdot h \grad_x$ and $|\g(x)|^2$ coincide with the coordinate representations of $\grad_{S^n} \log f \cdot h \grad_{S^n}$ and $|\grad_{S^n} \log f|^2$, respectively, in a neighbourhood of $\Omt$.
Then
\[
h^{\half}\|w\|_{L^2(\tilde{\Gamma}^c)} + \frac{h}{\sqrt{\e}}\|w\|_{H^1(\Omt)} \lesssim \|\Lphet w \|_{L^2(\Omt)} + h^{\half} \|h\grad_x w\|_{L^2(\tilde{\Gamma}^c)}
\]
where
\[
\Lphet = (1+|\g|^2)h^2\partial_r^2 - \frac{2}{r}(\a + \b \cdot h\grad_x)h\partial_y + \frac{1}{r^2}(\a^2 + h^2 \L)
\]
and $\L$ is a second order differential operator in the $x$ variables only, whose coefficients, in a neighbourhood of $\Omt$, coincide with the coordinate representation of $\Lap_{S^n}$.  
\end{prop}

Now we will restrict the graph conditions on $f$.  Let $\d > 0$ and $K \in \Rn$ be constants.  Suppose that the original domain $\Om$ is such that $\grad_{S^n} \log f $ is nearly constant, and in our choice of spherical coordinates, $\th_j$ are all near $\frac{\pi}{2}$.  Then we can suppose that $|\b - K| < \d, |\g - K| < \d,$ and $\|(\L - \Lap_x) v \|_{L^2(\Rn)} < \d \|v\|_{H^2(\Rn)}$ for $v \in H^2(\Rn)$.  

Now as in the linear case, we can choose $m_2 > m_1 > 0$, and $\mu_1$ and $\mu_2$ such that 
\[
\frac{|K|}{\sqrt{1+|K|^2}} < \mu_1 < \mu_2 < \half + \frac{|K|}{2\sqrt{1+|K|^2}} < 1,
\]
and $\rho \in \Czinf(\Rn)$ such that $\rho(\xi) = 1$ if $|\xi| < \mu_1$ and $|K \cdot \xi| < m_1$, and $\rho(\xi) = 0$ if $|\xi| > \mu_2$ or $|K \cdot \xi| > m_2$.  Again we can consider $w \in C^{\infty}(\Omt)$ such that $w \equiv 0$ in a neighbourhood of $\tilde{\Gamma}$, and $w$ satisfies \eqref{tildeBC}.  We extend $w$ by zero to become an element of $\S(\Rniip)$, use $\rho$ to split $w$ into a small frequency part $w_s$ and a large frequency part $w_{\ell}$, and write out the estimate from Proposition \ref{logtildepreCarl} as 
\begin{equation}\label{logfreqsplit}
h^{\half}\|w\|_{L^2(\Rnone)} + \frac{h}{\sqrt{\e}}\|w\|_{H^1(\Rniip)} \lesssim \|\Lphet w \|_{L^2(\Rniip)} + h^{\half} \|w_s\|_{\dot{H}^1(\Rnone)} + h^{\half} \|w_{\ell}\|_{\dot{H}^1(\Rnone)}.
\end{equation}

\begin{prop}\label{logsmallprop}
Suppose $w$ is as above.  There exist choices of $m_1, m_2, \mu_1,$ and $\mu_2$ depending only on $K$, such that if $\d$ is small enough,
\[
h^{\half}\|w_s\|_{\dot{H}^1(\Rnone)} \lesssim \|\Lphet w \|_{L^2(\Rniip)} + h\|w\|_{H^1(\Rniip)}.
\]
\end{prop}

\begin{proof}
We prove this by following the proof for the linear case.  For $V_1, V_2 \in \Rn$, define 
\[
A_{\pm}(V_1, V_2, \xi) = \frac{1 + iV_1 \cdot \xi \pm \sqrt{(1+iV_1\cdot \xi)^2 - (1+|V_2|^2)(1-|\xi|^2)}}{1 + |V_2|^2},
\]
so $A_{\pm}(V_1, V_2, \xi)$ are the roots of the polynomial
\[
(1 + |V_2|^2)X^2 - 2(1+ iV_1 \cdot \xi)X + (1 - |\xi|^2).
\]
As in the linear case, we can pick $F_{\pm}$ so that $F_{\pm} = A_{\pm}(K,K,\xi)$ on the support of $\hat{w}_s$, and $|F|, \mathrm{Re} \, F \simeq 1 + |\xi|$.  Then  
\begin{eqnarray*}
\|\Lphet w_s\|_{L^2(\Rniip)} &\geq& \left\| \left( (1+|K|^2)h^2\partial_r^2 - \frac{2}{r}(1 + K \cdot h\grad_x)h\partial_r + \frac{1 + h^2\Lap_x}{r^2} \right) w_s \right\|_{L^2(\Rniip)} \\
                             & &    - C\delta \|w_s\|_{H^2(\Rniip)}\\
\end{eqnarray*}
Now the operator in the first term does not factor exactly into $(h\partial_r - r^{-1}T_{F_{+}})(h\partial_r - r^{-1}T_{F_{-}})$ since the $r^{-1}T_{F_{\pm}}$ terms have $r$ dependence.  However, the errors can be bounded by $h\|w_s\|_{H^1(\Rniip)},$ so we can still get
\[
\|\Lphet w_s\|_{L^2(\Rniip)} \geq \|(h\partial_r - r^{-1}T_{F_{+}})(h\partial_r - r^{-1}T_{F_{-}})w_s \|_{L^2(\Rniip)} - C\delta\|w_s\|_{H^2(\Rniip)}.
\]
Then
\[
\|\Lphet w_s\|_{L^2(\Rniip)} \gtrsim \|(h\partial_r - r^{-1}T_{F_{-}})w_s \|_{H^1_r(\Rniip)} - C\delta\|w_s\|_{H^2(\Rniip)}.
\]
by Lemma \ref{logbounds}.  Since $r$ is bounded above and below by a constant on $\Omt$, and hence on the support of $w_s$, 
\[
\|\Lphet w_s\|_{L^2(\Rniip)} \gtrsim \|(h\partial_r - r^{-1}T_{F_{-}})w_s \|_{H^1(\Rniip)} - C\delta\|w_s\|_{H^2(\Rniip)}.
\]
Similarly,
\begin{eqnarray*}
\|\Lphet w_s\|_{L^2(\Rniip)} &\gtrsim& \|w_s \|_{H^2(\Rniip)} - C\delta\|w_s\|_{H^2(\Rniip)} \\
                             &\gtrsim& \|w_s \|_{H^2(\Rniip)} \\
\end{eqnarray*}
for small enough $\d$.  Then
\[
\|\Lphet  w\|_{L^2(\Rniip)} + h\|w\|_{H^1(\Rniip)} \gtrsim h^{\half}\|w_s\|_{\dot{H}^1(\Rnone)}
\]
as before.  
\end{proof}

\begin{prop}\label{loglargeprop}
Suppose $w$ and $w_{\ell}$ are as above.  Then if $\d$ is small enough,
\[
h^{\half} \|w_{\ell}\|_{\dot{H}^1(\Rnone)} \lesssim \|\Lphet w \|_{L^2(\Rniip)} + h\|w\|_{H^1(\Rniip)} + h^{\frac{3}{2}}\|w\|_{L^2(\Rnone)}.
\]
\end{prop}

\begin{proof}

Define $G^{\e}_{\pm}(V_1, V_2,\xi)$ and $G_{\pm}(V_1, V_2,\xi)$ in relation to $A_{\pm}(V_1,V_2, \xi)$ as for the linear case.  Then as before, we get
\begin{eqnarray*}
\|\Lphet w_{\ell} \|_{L^2(\Rniip)} &\gtrsim& \|(h\partial_r - r^{-1}T_{G_{+}(K,K, \xi)})(h\partial_r - r^{-1}T_{G^{\e}_{-}(\b,\g, \xi)})w_{\ell}\|_{L^2(\Rniip)} \\
                                   & &       - h\|w_{\ell}\|_{H^1(\Rniip)} - \d\|(h\partial_r - r^{-1}T_{G^{\e}_{-}(\b,\g, \xi)})w_{\ell}\|_{H^1(\Rniip)} .\\
                                   &\gtrsim& \|(h\partial_r - r^{-1}T_{G^{\e}_{-}(\b,\g, \xi)})w_{\ell}\|_{H^1_r(\Rniip)} \\
                                   & &       - h\|w_{\ell}\|_{H^1(\Rniip)} - \d\|(h\partial_r - r^{-1}T_{G^{\e}_{-}(\b,\g, \xi)})w_{\ell}\|_{H^1(\Rniip)} .\\
                                   &\gtrsim& \|(h\partial_r - r^{-1}T_{G^{\e}_{-}(\b,\g, \xi)})w_{\ell}\|_{H^1(\Rniip)} \\
                                   & &       - h\|w_{\ell}\|_{H^1(\Rniip)}.\\
\end{eqnarray*}
for small enough $\d$.  Then by the trace formula,
\[
\|\Lphet w_{\ell} \|_{L^2(\Rniip)} + h\|w_{\ell}\|_{H^1(\Rniip)} \gtrsim h^{\half}\|(h\partial_r - T_{G^{\e}_{-}(\b,\g, \xi)})w_{\ell}\|_{L^2(\Rnone)}.
\]
Using the boundary condition from \eqref{logtildeBC} gives
\[
\|\Lphet w_{\ell} \|_{L^2(\Rniip)} + h\|w_{\ell}\|_{H^1(\Rniip)} \gtrsim h^{\half}\|w_{\ell}\|_{\dot{H}^1(\Rnone)} - h^{\frac{3}{2}}\|w\|_{L^2(\Rnone)}.
\]
Therefore, following the linear case again, we get
\[
h^{\half} \|w_{\ell}\|_{\dot{H}^1(\Rnone)} \lesssim \|\Lphet w \|_{L^2(\Rniip)} + h\|w\|_{H^1(\Rniip)} + h^{\frac{3}{2}}\|w\|_{L^2(\Rnone)}
\]
as desired.

\end{proof}

Putting together \eqref{logfreqsplit}, Proposition \ref{logsmallprop}, and Proposition \ref{loglargeprop}, plus a change of variables, now gives Theorem \ref{mainCarl} for the logarthmic case, $\ph = \log r$, in the graph case, where $\Om$ is such that  $|\b - K| < \d, |\g - K| < \d,$ and $\|(\L - \Lap_x) v \|_{L^2(\Rn)} < \d \|v\|_{H^2(\Rn)}$ for $v \in H^2(\Rn)$.  Then these estimates can be glued together as in Section 7 to give Theorem \ref{mainCarl} for logarithmic case without graph conditions.  

Now the result for $\ph = -\log r$ can be obtained from the result for $\ph = \log r$ by using the change of variables $(r, \th) \mapsto (r^{-1}, \th)$.  Alternatively, using $\ph = -\log r$ and flipping signs as necessary in the proof above gives the desired result.  

Note that in general, if we were to replace $\ph$ with $-\ph$, then the sets $\partial \Om_{+}$ and $\partial \Om_{-}$ would reverse roles, so $Z$ would take the role of $\Gamma$ and $Z^c$ would take the role of $\Gamma^c$.  Then we would end up with a proof of the following result, which we will state as a corollary.  

\begin{cor}\label{reverseCarl}
If $\ph$ is as in Theorem \ref{mainCarl}, and $w \in H^2(\Om)$ with 
\begin{equation}\label{reverseBC}
\begin{split}
w, \partial_{\nu} w &= 0 \mbox{ on } Z \\
h\partial_{\nu} (e^{\frac{\ph}{h}} w) &= 0 \mbox{ on } Z^c,
\end{split}
\end{equation}
then
\begin{equation*}
h^{\half} \|w\|_{H^1(Z^c)} + h\|w\|_{H^1(\Om)}  \lesssim \| \L_{q,-\ph}w \|_{L^2(\Om)}
\end{equation*}
\end{cor}

%% file: 10CGOsolns.tex
%log version
Now we turn to the proof of Proposition \ref{boundarysolns}.  First we'll need a lemma.

\begin{lemma}\label{HBsolns}
For every $v \in H^{-1}(\RnI)$ and $f \in L^2(\partial \Om)$, there exists $u \in L^2(\Om)$ such that 
\begin{eqnarray*}
\Lphq  u &=& v  \mbox{ on } \Om\\
(h \partial_{\nu} - \partial_{\nu} \ph)u|_{Z^c} &=& f 
\end{eqnarray*}
and
\[
\|u\|_{L^2(\Om)} \lesssim h^{-1}\|v\|_{H^{-1}(\RnI)} + h^{\half}\|f\|_{L^2(\partial \Om)}.
\] 
\end{lemma}

\begin{proof}
We follow the methods in ~\cite{KSU}, but using the Carleman estimate from Corollary \ref{reverseCarl}.  Let $v \in H^{-1}(\RnI)$ and $f \in L^2(\partial \Om)$.  Suppose $w \in H^2(\Om)$ satisfies \eqref{reverseBC}, and consider the expression
\[
(w,v)_{\Om} + (w,hf)_{\partial \Om}.
\]
We have
\begin{equation*}
\begin{split}
|(w,v)_{\Om} + (w,hf)_{\partial \Om}| &\leq h\|w\|_{H^1(\RnI)}h^{-1}\|v\|_{H^{-1}(\RnI)} + h^{\half}\|w\|_{L^2(Z^c)}h^{\half}\|f\|_{L^2(\partial \Om)} \\
                                     &\lesssim \| \L_{\overline{q},-\ph}w \|_{L^2(\Om)}(h^{-1}\|v\|_{H^{-1}(\RnI)}+h^{\half}\|f\|_{L^2(\partial \Om)}), \\
\end{split}
\end{equation*}
with the second inequality being a consequence of Corollary \ref{reverseCarl}.
Now consider the subspace
\[
\{ \L_{\overline{q},-\ph}w | w \in H^2(\Om) \mbox{ and } w \mbox{ satisfies \eqref{reverseBC} } \} \subset L^2(\Om).
\]
By Corollary \ref{reverseCarl}, the linear functional $\L_{\overline{q},-\ph}w \mapsto (w,v)_{\Om} + (w,hf)_{\partial \Om}$ is well defined on this space.  Then the above estimate shows that it is bounded by $C(h^{-1}\|v\|_{H^{-1}(\RnI)}+h^{\half}\|f\|_{L^2(\partial \Om)})$.  Therefore by Hahn-Banach, there is an extension of the functional to the whole space $L^2(\Om)$ with the same bound.  This can be represented by an element of $L^2(\Om)$, so there exists $u \in L^2(\Om)$ such that 
\[
\|u\|_{L^2(\Om)} \lesssim h^{-1}\|v\|_{H^{-1}(\RnI)}+h^{\half}\|f\|_{L^2(\partial \Om)},
\]
and
\[
(w,v)_{\Om} + (w,hf)_{\partial \Om} = (\L_{\overline{q},-\ph}w, u).
\]
Integrating by parts on the right side,
\[
(w,v)_{\Om} + (w,hf)_{\partial \Om} = (w, \Lphq u)_{\Om} -h(h\partial_{\nu} w, u)_{\partial \Om} + h(w, h \partial_{\nu} u)_{\partial \Om} - 2h(w, \partial_{\nu} \ph u)_{\partial \Om}.
\]
This holds for all $w \in H^2(\Om)$ which satisfy \eqref{reverseBC}, so in particular it holds for all $w \in C^{\infty}_0(\Om)$.  This means that
\[
\Lphq u = v
\]
on $\Om$.  Thus
\[
(w,hf)_{\partial \Om} = -h(h\partial_{\nu} w, u)_{\partial \Om} + h(w, (h \partial_{\nu} - 2\partial_{\nu} \ph) u)_{\partial \Om}.
\]
Using the boundary conditions \eqref{reverseBC}, we get
\[
(w,hf)_{Z^c} = h(w, (h \partial_{\nu} - \partial_{\nu} \ph) u)_{Z^c}.
\]
for all $w \in H^2(\Om)$ which satisfy \eqref{reverseBC}.  Therefore 
\[
(h \partial_{\nu} - \partial_{\nu} \ph) u|_{Z^c} = f.
\]

\end{proof}

Now we can construct the CGO solutions.  

\begin{proof}[Proof of Proposition \ref{boundarysolns}]

If $\psi(x,y)$ solves the eikonal equations
\[
\grad \ph \cdot \grad \psi = 0,  |\grad \ph| = |\grad \psi|,
\]
and $a$ is a solution to the Cauchy-Riemann equation
\[
(-\grad \ph + i\grad \psi) \cdot \grad a + \half a \Lap (-\ph + i\psi) = 0
\]
then 
\[
h^2(-\Lap + q)e^{\frac{1}{h}(-\ph + i\psi)}a = O(h^2)e^{\frac{-\ph}{h}}.
\]
Now consider the problem
\begin{equation*}
\begin{split}
\grad \ell \cdot \grad \ell &= 0 \\
                \ell|_{Z^c} &= -\ph + i\psi.
\end{split}
\end{equation*}
By using power series and Borel's theorem, we can construct an approximate solution $\ell$ for $x$ near $Z^c$ with 
\begin{equation*}
\begin{split}
\grad \ell \cdot \grad \ell &= O(\mathrm{dist}(Z^c, x)^{\infty}) \\
                    \ell|_{Z^c} &= -\ph + i\psi \\
              \nd \ell|_{Z^c} &= -\nd(-\ph + i\psi)|_{Z^c}. \\
\end{split}
\end{equation*}
This last property also means that 
\begin{equation}\label{ellgood}
-(\mathrm{Re} \, \ell(x,y) + \ph(x,y)) \simeq \mathrm{dist}(Z^c, x)
\end{equation}
in a neighbourhood of $Z^c$.  (Similar constructions are used in ~\cite{KSU} and ~\cite{C}.)  

Then we can similarly create an approximate solution for the problem
\begin{equation*}
\begin{split}
\grad \ell \cdot \grad b - \half b \Lap \ell  &= 0 \\
                                  b|_{Z^c} &= a|_{Z^c},\\
\end{split}
\end{equation*}
so
\begin{equation*}
\begin{split}
\grad \ell \cdot \grad b - \half b \Lap \ell &= O(\mathrm{dist}(Z^c, x)^{\infty}) \\
                                  b|_{Z^c} &= a|_{Z^c},\\
\end{split}
\end{equation*}
Multiplying $b$ by a cutoff function which vanishes away from $Z^c$ does not change these properties, so we may as well cut off $b$ to have support inside a neighbourhood in which \eqref{ellgood} holds.  
Then
\[
h^2 (-\Lap+q)(e^{\frac{\ell}{h}}b) = e^{\frac{\ell}{h}} (O(\mathrm{dist}(x,E)^{\infty}) + O(h^2)),
\]
so
\begin{equation*}
\begin{split}
|h^2 (-\Lap+q)(e^{\frac{\ell}{h}}b)| &= e^{\frac{-\ph}{h}}e^{\frac{\mathrm{Re} \ell + \ph}{h}}(O(\mathrm{dist}(x,E)^{\infty}) + O(h^2)) \\
                                &= e^{\frac{-\ph}{h}}O(h^2). \\
\end{split}
\end{equation*}
Therefore
\[
e^{\frac{\ph}{h}}h^2(-\Lap + q)e^{\frac{-\ph+i\psi}{h}}(a + e^{\frac{\ell + \ph - i\psi}{h}}b) = v
\]
for some function $v$ with $\|v\|_{L^2(\Om)} = O(h^2)$, and
\[
e^{\frac{\ph}{h}} h \partial_{\nu} e^{\frac{-\ph+i\psi}{h}}(a + e^{\frac{\ell + \ph - i\psi}{h}}b)|_{Z^c} = g 
\]
for some function $g$ with $\|g\|_{L^2(\partial \Om)} = O(h)$.
By Lemma \ref{HBsolns}, the problem 
\begin{eqnarray*}
\Lphq r_0 &=& -v  \mbox{ on } \Om\\
(h \partial_{\nu} - \partial_{\nu} \ph)r_0|_{Z^c} &=& -g 
\end{eqnarray*}
has a solution $r_0 \in L^2(\Om)$ with $\|r_0\|_{L^2(\Om)} = O(h)$.  Then if $r_1 = e^{\frac{-i\psi}{h}}r_0$, then $\|r_1\|_{L^2(\Om)} = O(h)$, and 
\[
(-\Lap + q)e^{\frac{-\ph+i\psi}{h}}(a + e^{\frac{\ell + \ph - i\psi}{h}}b + r_1) = 0
\]
on $\Om$, with 
\[
\partial_{\nu} e^{\frac{-\ph+i\psi}{h}}(a + e^{\frac{\ell + \ph - i\psi}{h}}b + r_1)|_{Z^c} = 0
\]
Now note that \eqref{ellgood} implies that $\|e^{\frac{\ell + \ph - i\psi}{h}}b\|_{L^2(\Om)} = O(h^{\half}).$  Therefore we can let $r = e^{\frac{\ell + \ph - i\psi}{h}}b + r_1$, and then
\[
u = e^{\frac{-\ph+i\psi}{h}}(a + r)
\]
is the desired solution.  In the logarithmic case, $a$ and $\psi$ are exactly the functions used in the CGO solutions in ~\cite{KSU}.  In the linear case, we can use $a = 1$, and $\psi$ is exactly as in the CGO solutions in ~\cite{BU}.

\end{proof}

%% file: 11Refs.tex
%ND log version